\documentclass{amsart}

\newcommand{\apref}[3]{\hyperref[#2]{#1\ref*{#2}#3}}

\usepackage{enumerate}
\usepackage[latin1]{inputenc}
\usepackage{dsfont}
\usepackage{amssymb,amsthm,amsmath}

\input{xy}
\xyoption{all}

\usepackage{mathrsfs}

\theoremstyle{plain}
\newtheorem{prop}{Proposition}[section]
\newtheorem{lemma}[prop]{Lemma}

\newtheorem{thm}[prop]{Theorem}

\newtheorem{cor}[prop]{Corollary}

\theoremstyle{definition}

\newtheorem{defi}[prop]{Definition}

\theoremstyle{remark}

\newcommand{\mapeins}{\phi_1}
\newcommand{\mapeinsn}{\phi_{1,n}}
\newcommand{\mapeinseins}{\phi_{1,1}}
\newcommand{\mapzwei}{\phi_2}
\newcommand{\mapzwein}{\phi_{2,n}}
\newcommand{\mgf}{\tau}
\newcommand{\mapdrei}{\psi_1}
\newcommand{\mapdrein}{\psi_{1,n}}
\newcommand{\mapvier}{\psi_2}
\newcommand{\mapviern}{\psi_{2,n}}

\DeclareMathOperator{\discr}{discr}
\DeclareMathOperator{\sign}{sign}
\DeclareMathOperator{\bdinterval}{bdinterval}
\DeclareMathOperator{\interval}{interval}

\DeclareMathOperator{\UGeo}{UGeo}
\DeclareMathOperator{\PGeo}{PGeo}

\DeclareMathOperator{\FP}{FP}
\DeclareMathOperator{\POQF}{POQF}
\DeclareMathOperator{\OQF}{OQF}

\DeclareMathOperator{\base}{base}

\DeclareMathOperator{\QF}{QF}

\DeclareMathOperator{\SL}{SL}
\DeclareMathOperator{\PSL}{PSL}

\DeclareMathOperator{\tr}{tr}

\DeclareMathOperator{\Ima}{Im}
\DeclareMathOperator{\Rea}{Re}
\DeclareMathOperator{\sgn}{sgn}

\DeclareMathOperator{\pr}{pr}

\DeclareMathOperator{\Stab}{Stab}

\newcommand{\st}{\text{st}}

\newcommand\N{\mathbb{N}}

\newcommand\R{\mathbb{R}}
\newcommand\Z{\mathbb{Z}}
\newcommand\C{\mathbb{C}}

\newcommand{\h}{\mathbb{H}}

\newcommand{\mc}[1]{\mathcal #1}

\newcommand{\wt}{\widetilde}
\newcommand{\wh}{\widehat}

\newcommand{\eps}{\varepsilon}

\DeclareMathOperator{\id}{id}

\newcommand{\sceq}{\mathrel{\mathop:}=}
\newcommand{\seqc}{\mathrel{=\mkern-4.5mu{\mathop:}}}

\newcommand{\mat}[4]{\begin{pmatrix} #1&#2\\#3&#4\end{pmatrix}}
\newcommand{\bmat}[4]{\begin{bmatrix} #1&#2\\#3&#4\end{bmatrix}}
\newcommand{\textmat}[4]{\left(\begin{smallmatrix} #1&#2 \\ #3&#4
\end{smallmatrix}\right)}
\newcommand{\textbmat}[4]{\left[\begin{smallmatrix} #1&#2 \\ #3&#4
\end{smallmatrix}\right]}

\usepackage[colorlinks,breaklinks]{hyperref}
\usepackage{graphics}
\usepackage[T1]{fontenc} 
\usepackage{pdflscape}
\allowdisplaybreaks

\usepackage[ruled]{algorithm}
\usepackage{algpseudocode}

\begin{document}

\title[Reduction theory for quadratic forms over $\Z{[}\lambda{]}$]{A geometric reduction theory for indefinite binary quadratic forms over $\Z[\lambda]$}
\author[A.\@ Pohl]{Anke D.\@ Pohl}
\address{Max Planck Institute for Mathematics, Vivatsgasse 7,  D-53119 Bonn}
\email{pohl@mpim-bonn.mpg.de}
\author[V.\@ Spratte]{Verena Spratte}
\address{Mathematisches Institut, Georg-August-Universit\"at G\"ottingen,  Bunsenstr. 3-5, D-37073 G\"ottingen}
\email{verena.spratte@stud.uni-goettingen.de}
\begin{abstract} 
Gauss' classical reduction theory for indefinite binary quadratic forms over $\Z$ has originally been proven by means of purely algebraic and arithmetic considerations. It was later discovered that this reduction theory is closely related to a certain symbolic dynamics for the geodesic flow on the modular surface, and hence can also be deduced geometrically. In this article, we use certain symbolic dynamics for the geodesic flow on Hecke triangle surfaces (also the non-arithmetic ones) to develop reduction theories for the indefinite binary quadratic forms associated to Hecke triangle groups. Moreover, we propose an algorithm to decide for any $g\in \PSL_2(\R)$ whether or not $g$ is contained in the Hecke triangle group under consideration, and provide an upper estimate for its run time.
\end{abstract}
\maketitle

\section{Introduction}

Quadratic forms arise in many contexts and have numerous applications. Since early on, reduction theories play an important role in their study. Gauss---in his famous opus \textit{Disquisitiones Arithmeticae} \cite{Gauss}---developed a reduction theory for indefinite binary quadratic forms over $\Z$ which ever since has been extremely influential. His proof was based on the algebraic and arithmetic properties of these quadratic forms. Hurwitz \cite{Hurwitz_cf, Hurwitz_red} provided an alternative geometric and dynamical proof of this reduction theory by modelling the quadratic forms as objects on the hyperbolic plane and the steps in the reduction algorithm as the action of Riemannian isometries. For a modern account in terms of discretizations for the geodesic flow on the modular surface, symbolic dynamics and continued fractions we refer to \cite{Artin, Series}.

For various hyperbolic surfaces and two-dimensional good hyperbolic orbifolds $X$, the first-named author developed discretizations and symbolic dynamics for the geodesic flow on $X$ with certain special properties that enabled dynamical characterizations of the Maass cusp forms for the associated Fuchsian groups \cite{Pohl_diss, Pohl_Symdyn2d, Moeller_Pohl, Pohl_gamma, Pohl_mcf_general, Pohl_spectral_hecke}. In this article, we show that for any cofinite Hecke triangle surface, these discretizations and symbolic dynamics induce a reduction theory for the indefinite binary quadratic forms associated to the Hecke triangle group under consideration. The techniques we use are uniform for the arithmetic and the non-arithmetic Hecke triangle groups. For the modular group $\PSL_2(\Z)$, we essentially recover Gauss' reduction theory. This two-fold application of the symbolic dynamics alludes at intimate relations between Maass cusp forms and indefinite binary quadratic forms for all Hecke triangle groups. Moreover, 
we expect that the 
same techniques yield reduction theories for 
indefinite binary quadratic forms for other `admissible' Fuchsian lattices as well.

The coefficients of any indefinite binary quadratic form associated to a Hecke triangle group $\Gamma$ are contained in $\Z[\lambda]$ (hence the title), where $\lambda$ is the cusp width of $\Gamma$, thus
\[
 \lambda = 2\cos\frac{\pi}{q} 
\]
for some $q\in\N$, $q\geq 3$. For $\lambda >1$ (i.\,e., for $q>3$), not every indefinite binary quadratic form over $\Z[\lambda]$ is associated to $\Gamma$. We provide a counterexample in Section~\ref{whenHecke} below. Any fundamental domain for $\Gamma$ in the hyperbolic plane $\h$ yields an algorithm for deciding whether or not any given element $g\in\PSL_2(\R)$ belongs to $\Gamma$, which \textit{a fortiori} overcomes the decision problem for quadratic forms. For the arithmetic Hecke triangle groups, algorithms of another type allow us to decide immediately which elements of $\PSL_2(\R)$ belong to them. For the non-arithmetic Hecke triangle groups such a decision algorithm with uniformly bounded run time is not known yet. Our techniques for the reduction theory enable us to develop an algorithm which takes advantage of a fundamental set for $\Gamma$ in the unit tangent bundle $S\h$ and for which we can provide an upper estimate for its run time. We present the details in Section~\ref{whenHecke} below.

The reduction theory for the quadratic forms associated to $\Gamma$ makes use of the standard relations between these forms, the hyperbolic matrices in $\Gamma$, the periodic geodesics on the Hecke triangle surface $\Gamma\backslash\h$, and the lifts of these geodesics to $\h$. For the convenience of the reader and to obtain a certain level of selfcontainedness of this article, we provide a detailed account on these relations in Sections~\ref{sec:relations} and \ref{sec:rel_hecke}, after introducing the main objects in Section~\ref{sec:prelims}. The reduction algorithm, which we present and prove in Section~\ref{sec:redtheory}, not only finds all reduced forms that are equivalent to any given quadratic form but also the associated hyperbolic matrices. An alternative reduction theory is proposed by Ressler \cite{Ressler_reduction} based on Rosen's $\lambda$-continued fractions.

\section{Preliminaries}\label{sec:prelims}

\subsection{Quadratic forms}

A \textit{binary quadratic form} over $\R$ is a homogeneous polynomial (function) 
\[
f(x,y) = Ax^2 + Bxy + Cy^2
\]
of degree $2$ with real coefficients. We denote it by
\[ 
 f = [A,B,C].
\]
We call $f$ \textit{indefinite} if its discriminant 
\[ 
 D \sceq D(f) \sceq B^2 - 4AC
\]
is positive.  Let 
\[
\QF \sceq \QF(\R) \sceq \{ f=[A,B,C] \mid D>0 \}
\]
be the set of indefinite binary quadratic forms over $\R$. Since we restrict our considerations to indefinite binary quadratic forms throughout this article, we refer to the elements of $\QF$ as quadratic forms only. 

To each $f=[A,B,C] \in \QF$ we associate the two values
\[
 x_+ \sceq x_+(f) \sceq \frac{-B + \sqrt{D}}{2A} \quad\text{and}\quad x_- \sceq x_-(f) \sceq \frac{-B-\sqrt{D}}{2A},
\]
which we call the \textit{zeros} of $f$. If $A=0$, hence $D=B^2$, then these are to be understood as
\[
\begin{cases}  x_+  = 0 \\  x_- = \infty \end{cases}  \text{ if $B>0$, and}\quad  
\begin{cases} x_+  = \infty \\ x_- = 0 \end{cases} \text{ if $B<0$.}
\]
For $A\not=0$, the values $x_\pm$ are the two zeros of $f(x,1)$.

\subsection{The hyperbolic plane and its geometry}\label{sec:hyperbolic}

\subsubsection*{The model}
Throughout we use the upper half plane model for the hyperbolic plane, that is, 
\[
 \h \sceq \{ z\in\C \mid \Ima z > 0\}
\]
endowed with the hyperbolic Riemannian metric given by the line element
\[
 ds^2 = \frac{dz d\overline{z}}{(\Ima z)^2}. 
\]
We identify its geodesic boundary with $P^1(\R) \cong \R \cup \{\infty\}$. The action of the group of Riemannian isometries on $\h$ extends continuously to $P^1(\R)$.

\subsubsection*{Riemannian isometries}
The group of orientation-preserving Riemannian isometries of $\h$ is isomorphic to 
\[
 G\sceq \PSL_2(\R) = \SL_2(\R)/\{\pm\id\}, 
\]
where the action of $G$ on $\h\cup P^1(\R)$ is given by fractional linear transformations. To be more explicit, we denote elements in $G$ as matrices but with square brackets. Hence if $g\in G$ is represented by $\textmat{a}{b}{c}{d}\in \SL_2(\R)$ then we write
\[
 g = \bmat{a}{b}{c}{d}.
\]
At any occasion where we use the `entries' $a,b,c,d$ of $g$ for some formula, it is easily seen that this formula is independent of the choice of the representative of $g$.

The action of $\textbmat{a}{b}{c}{d}\in G$ on $\h\cup P^1(\R)$ is given by
\begin{align*}
\bmat{a}{b}{c}{d}.z & = 
\begin{cases}
\frac{az+b}{cz+d} & \text{for $cz+d\not=0$}
\\
\infty & \text{for $cz+d=0$} 
\end{cases}
\quad\text{if $z\in \h\cup\R$,}
\intertext{and}
\bmat{a}{b}{c}{d}.\infty &= 
\begin{cases}
\frac{a}{c} & \text{for $c\not=0$}
\\
\infty & \text{for $c=0$.}
\end{cases}
\end{align*}

We denote the induced action of $G$ on the unit tangent bundle $S\h$ of $\h$ by $g.v$ for $g\in G$ and $v\in S\h$.

\subsubsection*{Geodesics}

Let $\UGeo$ denote the set of \textit{(unit speed) geodesics} on $\h$, that is, the geodesics parametrized by arc length. One such geodesic is 
\[
 \gamma_\st \colon \R \to \h,\quad t\mapsto ie^t,
\]
which we call the \textit{standard geodesic}. By the very definition of Riemannian isometries, the pointwise action of $G$ maps geodesics to geodesics. It is well-known that this action is simple transitive.  

\begin{prop}\label{twopoint}
For every geodesic $\gamma\in\UGeo$ there exists a unique element $g\in G$ such that $g.\gamma = \gamma_\st$. 
\end{prop}

Proposition~\ref{twopoint} implies that the images of the geodesics on $\h$ are exactly the semicircles that are orthogonal to $\R$, and all vertical lines (i.\,e., semicircles through $i\infty$).

For $\gamma\in\UGeo$ we set
\[
 \gamma(\pm\infty) \sceq \lim_{t\to\pm\infty} \gamma(t) \quad\in P^1(\R).
\]

For any $v\in S\h$ we let $\gamma_v$ denote the geodesic on $\h$ determined by $v$, that is,
\[
 \gamma'_v(0) = v.
\]
E.\,g., the standard geodesic $\gamma_\st$ is determined by $v_\st=\frac{\partial}{\partial y}\vert_i$. The geodesic flow on $\h$ is then given by
\[
 \R\times S\h \to S\h,\quad (t,v) \mapsto \gamma_v'(t).
\]
Finally, for any $a,b\in\h\cup P^1(\R)$, $a\not=b$, we let $[a,b]_\gamma$ denote the (non-oriented) \textit{geodesic arc} connecting $a$ and $b$ and containing both endpoints. We use $[a,b)_\gamma$ to denote the geodesic arc containing $a$ but not $b$, and similarly use $(a,b]_\gamma$ and $(a,b)_\gamma$. We call a geodesic arc \textit{maximal} if it is the image of a complete geodesic.

\subsubsection*{Hyperbolic matrices}

An element of $G$ is called \textit{hyperbolic} if it fixes exactly two points in $\h\cup P^1(\R)$. In that case both fixed points are in $P^1(\R)$. As it is well-known, $\textbmat{a}{b}{c}{d}\in G$ is hyperbolic if and only if $|a+d|>2$. Let
\[
 G_h \sceq \{g\in G \mid \text{$g$ hyperbolic}\}
\]
denote the set of hyperbolic elements in $G$.

Let $g\in G_h$ and suppose that $g$ fixes the two points $w_1,w_2\in P^1(\R)$. One of these two fixed points is \textit{attracting}, say $w_1$, the other one \textit{repelling}. This means that for all $z\in\h$ we have
\begin{equation}\label{def_attracting}
  \lim_{n\to\infty} g^n.z = w_1 \quad\text{and}\quad \lim_{n\to\infty} g^{-n}.z=w_2.
\end{equation}
 
Suppose that $g=\textbmat{a}{b}{c}{d}$  and set
\begin{equation}\label{def_lambda}
 \lambda_{\pm}(g) \sceq \frac{|a+d| \pm \sqrt{(a+d)^2 -4}}{2}.
\end{equation}
For $c\not=0$ we associate to $g$ the two values
\begin{align}
\label{def_wa1} w_a &\sceq w_a(g) \sceq \frac{\lambda_+(g) - \sgn(a+d) \cdot d}{\sgn(a+d) \cdot c}
\intertext{and}
\label{def_wr1} w_r &\sceq w_r(g) \sceq \frac{\lambda_-(g) - \sgn(a+d)\cdot d}{\sgn(a+d) \cdot c}.
\end{align}
If $c=0$ then we associate to $g$ the values
\begin{equation}\label{def_w2}
 w_a \sceq \infty, \ w_r \sceq \frac{b}{a-a^{-1}} \quad\text{if $|a|>1$,}
\end{equation}
and
\begin{equation}\label{def_w3}
 w_a \sceq \frac{b}{a-a^{-1}},\ w_r \sceq \infty \quad\text{if $|a|<1$.}
\end{equation}

\begin{lemma}\label{attr_rep}
Let $g\in G_h$. Its attracting fixed point is $w_a(g)$, and its repelling fixed point is $w_r(g)$.
\end{lemma}

\begin{proof}
For $g=\textbmat{a}{b}{c}{d}$ with $c\not=0$, the statement is shown in \cite{Moeller_Pohl}. For $c=0$ it easily follows from $g'(z) = a^2$.
\end{proof}

\subsubsection*{Topological notations and an order on $P^1(\R)$} 
Let $M\subseteq \h$ be a subset. Throughout we use $M^\circ$ to denote the interior of $M$, $\overline{M}$ the closure of $M$ in $\h$, and $\partial M$ the boundary of $M$ in $\h$.

All our constructions are independent of the model we use for the hyperbolic plane. However, occasionally the exposition is less cumbersome if we rely on some special properties of $\h$ and $P^1(\R)$. In this spirit, we introduce an auxiliary order on $P^1(\R)$ via its identification with $\R\cup\{\infty\}$ and defining $r<\infty$ for all $r\in\R$.

\subsection{Cofinite Hecke triangle groups}\label{sec:hecketriangle}

Let $q\in\N$, $q\geq 3$, and set $\lambda_q\sceq 2\cos\tfrac{\pi}{q}$. The \textit{Hecke triangle group} $\Gamma_q$ with parameter $q$ is the lattice (i.e., discrete and cofinite subgroup) in $G$ which is generated by the two elements
\begin{equation}\label{generators}
 S \sceq \bmat{0}{1}{-1}{0}\quad\text{and}\quad T_q\sceq \bmat{1}{\lambda_q}{0}{1}.
\end{equation}
For $q=3$, the Hecke triangle group $\Gamma_3$ is the well-investigated modular group $\PSL_2(\Z)$. The family $\{ \Gamma_q \mid q\in\N,\ q\geq 3\} $ of Hecke triangle groups contains arithmetic lattices (for $q\in\{3,4,6\}$) as well as non-arithmetic ones (for $q\in\N$, $q\notin\{1,2,3,4,6\}$). However, as we will see, our techniques are uniform throughout. 

We denote the associated hyperbolic surface (more precisely, the real hyperbolic good orbifold) by
\[
 X_q\sceq\Gamma_q\backslash\h,
\]
that is, the orbit space of the left-action of $\Gamma_q$ on $\h$. We usually omit the subscript `$q$' if understood implicitly, thus
\[
 \lambda\sceq\lambda_q,\quad \Gamma\sceq \Gamma_q,\quad X\sceq X_q,\quad \ldots.
\]

We call a subset $\mc F$ of $\h$ a \textit{fundamental domain} for $\Gamma$ if $\mc F$ is open and connected, 
\[
 \mc F \cap g.\mc F = \emptyset \quad\text{for all $g\in\Gamma$, $g\not=\id$,}
\]
and 
\[
 \h = \bigcup_{g\in\Gamma} g.\overline{\mc F}.
\]
One example of a fundamental domain for $\Gamma$ is (see Figure~\ref{PS_funddom})
\begin{equation}\label{funddom}
 \mc F \sceq \left\{ z\in\h \left\vert\ |z|>1,\ |\Rea z| < \tfrac{\lambda}{2}\right.\right\}.
\end{equation}
\begin{center}
\begin{figure}[h]
\includegraphics{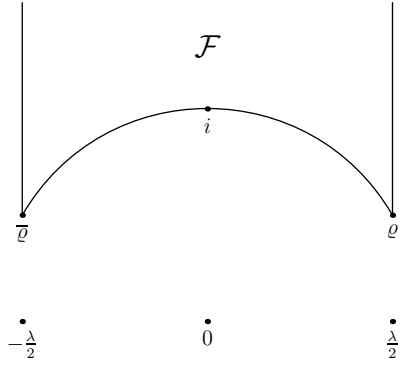} 
\caption{A fundamental domain for $\Gamma$ in $\h$}\label{PS_funddom}
\end{figure}
\end{center}
The side-pairings for $\mc F$ are given by the generators~\eqref{generators}: Let 
\[
 \varrho \sceq \frac{\lambda + i\sqrt{4-\lambda^2}}{2}.
\]
The vertical sides $[\overline\varrho,\infty)_\gamma$ and $[\varrho,\infty)_\gamma$ are mapped to each other by $T$, the bottom sides 
$[\overline\varrho,i]_\gamma$ and $[i,\varrho]_\gamma$ are identified by $S$. Hence, the Hecke triangle surface $X=\Gamma\backslash\h$ has one cusp (represented by $\infty$ on $\h$) and two singularity points (represented by $i$ and  $\varrho$).

Geometric entities on $X$ are defined as $\Gamma$-orbits ($\Gamma$-equivalence classes) of the corresponding entities on $\h$. For example, if $v\in S\h$ is a unit tangent vector of $\h$ then
\[
\wh v \sceq \Gamma.v
\]
is a \textit{unit tangent vector} of $X$, resulting in the unit tangent bundle of $X$ being
\[
 S(\Gamma\backslash\h) = \Gamma\backslash S\h.
\]
Likewise, if $\gamma\in \UGeo$ is a geodesic on $\h$, then 
\[
 \wh\gamma \sceq \Gamma.\gamma
\]
is a \textit{(unit speed) geodesic} on $X$. 

Throughout we denote the entities on $X$ by a representative on $\h$ endowed with  $\wh{\ \ }$. Alternatively, we use the canonical quotient maps
\[
\pi\colon \h \to \Gamma\backslash\h \qquad\text{resp.\@}\qquad \pi\colon S\h \to \Gamma\backslash S\h
\]
from $\h$ to $X$ respectively from $S\h$ to $SX$. Then the geodesic $\gamma\in\UGeo$ determines the geodesic $\wh\gamma = \pi(\gamma)$ on $X$ given by
\[
 \wh\gamma = \pi(\gamma)\colon \R \to X,\quad t\mapsto \pi(\gamma(t)) = \Gamma.\gamma(t)
\]
and
\[
 \wh\gamma'(t) = \pi(\gamma'(t)).
\]
If $\eta\in \UGeo$ represents the geodesic $\wh\gamma$ on $X$ then we call $\eta$ a \textit{lift} of $\wh\gamma$.

For $\wh v\in SX$ let $\wh\gamma_{\wh v}$ denote the geodesic on $X$ determined by $\wh v$, that is,
\[
 \wh\gamma_{\wh v}'(0) = \wh v.
\]
The geodesic flow on $X$ is then given by
\[
 \R\times SX \to SX,\quad (t,\wh v) \mapsto \wh\gamma_{\wh v}'(t).
\]
Whereas all maximal geodesic arcs on $\h$ are semicircles or vertical lines, the maximal geodesic arcs on $X$ enjoy a larger variety of possible forms. E.\,g., they can be dense, or vanish into the cusp, or being dense in some unbounded subset, or being bounded, and so on. We are mostly interested in \textit{periodic geodesics}, that is, geodesics $\wh\gamma$ for which there exists $t_0>0$ such that 
\[
 \wh\gamma'(0) = \wh\gamma'(t_0)
\]
or, equivalently,
\[
 \wh\gamma(t) = \wh\gamma(t+t_0)
\]
for all $t\in\R$. Let $\PGeo(\Gamma)$ denote the set of periodic geodesics on $X$.

An element $g=\textbmat{a}{b}{c}{d}\in G$ is called \textit{parabolic} if $|a+b|=2$. A point in $P^1(\R)$ is called \textit{cuspidal} if it is fixed by a parabolic element in $\Gamma$. The cuspidal points are precisely the representatives in $P^1(\R)$ of the cusps of $X$. Thus, since the Hecke triangle surface $X$ has only one cusp and this cusp is represented by $\infty$, the set of cuspidal points of $\Gamma$ is just $\Gamma.\infty$. If $\gamma$ is a lift of a periodic geodesic on $X$ then neither of its two endpoints $\gamma(\pm\infty)$ is cuspidal.

\section{Relations between quadratic forms, hyperbolic matrices, and geodesics}\label{sec:relations}

The reduction theory developed in this article takes advantage of equivariant identifications between quadratic forms, hyperbolic matrices and unit speed geodesics, or more precisely, between slight modifications of these objects. In this section we expound these modifications and relations on the level of the hyperbolic plane $\h$ and the full group $G=\PSL_2(\R)$.

The rough idea is to consider a quadratic form $f\in\QF$ with zeros $\{x_a(f),x_r(f)\} = \{x_\pm(f)\}$, a hyperbolic matrix $g\in G_h$ and a geodesic $\gamma$ on $\h$ as being related if and only if 
\begin{equation}\label{endident}
 x_a(f) = w_a(g) = \gamma(+\infty) \quad\text{and}\quad x_r(f) = w_r(g) = \gamma(-\infty).
\end{equation}
Moreover, $G$ naturally acts on the set of geodesics. The $G$-actions on quadratic forms and hyperbolic matrices are then uniquely defined by requiring $G$-equivariance of these relations.

One immediately sees that the relations between quadratic forms, hyperbolic matrices and geodesics on $\h$ as induced by \eqref{endident} are not bijections since several quadratic forms have identical zeros, several hyperbolic matrices identical fixed points, and several geodesics identical endpoints. Moreover, these relations involve a non-canonical choice as to whether $x_a(f) = x_+(f)$ or $x_a(f) = x_-(f)$. 

To overcome the latter issue we endow each quadratic form with two possible \textit{orientations}, thereby introducing a double cover $\QF\times\{\pm 1\} \to \QF$. To resolve the former issue we take advantage of additional properties of the involved objects. For example, the action of the matrix $g\in G_h$ on any geodesic $\gamma$ satisfying \eqref{endident} has the effect of a time-shift in the parametrization of $\gamma$ by a fixed positive increment, the \textit{displacement}, independent of the choice of $\gamma$. Thus, we need to consider geodesics only up to parametrization but endow them in addition with a displacement factor. For the oriented quadratic forms we need to decend to a projective version by factoring out common signs of their coefficients.

Throughout let 
\[
 \FP \sceq P^1(\R)^2 \setminus \{\text{diagonal}\} 
\]
be endowed with the diagonal $G$-action. The elements of $\FP$ serve as pairs of fixed points for hyperbolic matrices, of endpoints for geodesics, and of zeros for quadratic forms.

\subsection{Hyperbolic matrices and geodesics}

Recall from Lemma~\ref{attr_rep} that for any $g\in G_h$, the point $w_a(g)$ from \eqref{def_wa1} respectively \eqref{def_w2} is its attracting fixed point, and $w_r(g)$ from \eqref{def_wr1} respectively \eqref{def_w3} is its repelling fixed point.

\begin{lemma}\label{fixedconj}
Let $g\in G_h$ and $h\in G$. Then 
\[ 
 h.w_a(g)=w_a(hgh^{-1}) \quad\text{and}\quad h.w_r(g)=w_r(hgh^{-1}).
\]
\end{lemma}

\begin{proof}
Let $z\in\h$. Then
\begin{align*}
w_a(hgh^{-1}) & = \lim_{n\to\infty} (hgh^{-1})^n.z = \lim_{n\to\infty} hg^nh^{-1}.z = h.\big(\lim_{n\to\infty} g^n.(h^{-1}.z)\big)
\\
& = h.w_a(g), 
\end{align*}
where the first and the last equality is justified by Lemma~\ref{attr_rep}. 
\end{proof}

We define the left-action of $G$ on itself by conjugation, hence
\[
 h.g \sceq hgh^{-1}
\]
for all $h,g\in G$.  Then the map
\begin{equation}\label{phi1}
 \mapeins\colon G_h \to \FP,\ g\mapsto (w_a(g), w_r(g)),
\end{equation}
is $G$-equivariant.

\begin{prop}\label{displacement}
Let $\gamma\in \UGeo$. 
\begin{enumerate}[{\rm (i)}]
\item\label{displacegroup} Let $g\in G_h$ and suppose that 
\begin{equation}\label{sameend}
 w_a(g) = \gamma(+\infty) \quad\text{and}\quad w_r(g) = \gamma(-\infty).
\end{equation}
Then there exists $t_g>0$ such that $g.\gamma = \gamma(\cdot + t_g)$. 
\item\label{displaceindep} Let $\eta \in \UGeo$ be any geodesic with $\eta(\pm\infty) = \gamma(\pm\infty)$, and let $g\in G_h$ and $t_g>0$ be as in \eqref{displacegroup}. Then $g.\eta= \eta(\cdot + t_g)$. In other words, the value of $t_g$ does not depend on $\gamma$.
\item\label{displacevalue} Let $t\in \R_{>0}$. Then there exists a unique element $g\in G_h$ with \eqref{sameend} and $t_g=t$.
\item\label{displaceinv} For any $g\in G_h$ and $h\in G$ we have $t_{hgh^{-1}} = t_g$.
\end{enumerate}
\end{prop}

For $g\in G_h$, the value $t_g$ in Proposition~\ref{displacement} is called the \textit{displacement length} of $g$.

\begin{proof}[Proof of Proposition~\ref{displacement}]
We start by proving statements \eqref{displacegroup} and \eqref{displacevalue} for the standard geodesic 
\[
 \gamma_\st \colon \R \to \h,\quad t\mapsto ie^t.
\]
Since $\gamma_\st(+\infty) = \infty$ and $\gamma_\st(-\infty) = 0$, all elements $g\in G_h$ satisfying \eqref{sameend} are given by
\[
 g =  g_a = \bmat{a}{0}{0}{a^{-1}}
\]
with $a>1$. Since 
\[
 g_a.\gamma_\st(t) = ia^2 e^t = ie^{t+ 2\log a},
\]
we have $t_{g_a} = 2\log a > 0$. Moreover, for each given $t>0$ there is a unique $a>1$ such that $t= t_{g_a}$, namely $a = e^{t/2}$.

Now let $\eta\in \UGeo$ be a reparametrization of $\gamma_\st$. Hence, there exists $t_0\in\R$ such that $\eta = \gamma_\st(\cdot + t_0)$. Obviously, $g_a.\eta = \eta(\cdot + t_{g_a})$, which proves \eqref{displaceindep} for these cases.

Finally, let $\gamma\in \UGeo$ be arbitrary. By Proposition~\ref{twopoint} there exists a unique element $h\in G$ such that $h.\gamma = \gamma_\st$. Thus, if $g\in G_h$ satisfies \eqref{sameend} for $\gamma$, then $hgh^{-1}$ satisfies \eqref{sameend} for $\gamma_\st$. Hence the previous considerations for $\gamma_\st$ carry over and imply \eqref{displacegroup}-\eqref{displaceindep} for $\gamma$. In addition, this shows \eqref{displaceinv}. 
\end{proof}

We define $\gamma,\eta\in \UGeo$ to be \textit{equivalent} if $\eta$ is a reparametrization of $\gamma$. That is, there exists $t_0\in\R$ such that $\eta = \gamma(\cdot + t_0)$. We denote the equivalence class of $\gamma$ by $[\gamma]$, and the set of all equivalence classes by 
\[
 \UGeo/_\sim.
\]
Then $G$ acts on $\UGeo/_\sim$ by 
\[
 g.[\gamma] \sceq [g.\gamma]
\]
for all $g\in G$, $[\gamma] \in \UGeo/_\sim$.  The map
\[
  \mgf\colon \UGeo/_\sim \to \FP,\quad [\gamma] \mapsto (\gamma(+\infty), \gamma(-\infty)),
\]
is well-defined, a bijection and $G$-equivariant. Motivated by Proposition~\ref{displacement} we consider
\[
 \UGeo/_\sim \times \R_{>0} \cong \FP\times \R_{>0},
\]
where the action of $G$ is defined by restrictions to $\UGeo/_\sim$ and $\FP$.

\begin{thm}\label{mmg}
The map 
\[
 \mapdrei\colon G_h \to \FP\times\R_{>0},\ g\mapsto ( w_a(g), w_r(g), t_g )
\]
is a $G$-equivalent bijection.
\end{thm}

\begin{proof}
This follows immediately from Proposition~\ref{displacement}.
\end{proof}

\subsection{Quadratic forms}
Let 
\[
 \QF\times\{\pm1\} \to \QF,\quad (f,s) \mapsto f,
\]
be the trivial double cover of $\QF$ (endowed with the discrete topology). For $(f,s)\in\QF\times\{\pm 1\}$ we call
\[
 x_a\big( (f,s) \big) \sceq
\begin{cases}
x_+(f) & \text{if $s=+1$}
\\
x_-(f) & \text{if $s=-1$} 
\end{cases}
\]
the \textit{attracting zero} of $(f,s)$, and 
\[
 x_r\big( (f,s) \big)\sceq
\begin{cases}
x_-(f) & \text{if $s=+1$}
\\
x_+(f) & \text{if $s=-1$}
\end{cases}
\]
the \textit{repelling zero}. The value of $s$ should be understood as a choice of \textit{orientation} for the quadratic form $f$.  Let
\[ 
 \OQF \sceq \OQF(\R) \sceq \QF(\R) \times\{\pm 1\}
\]
denote the set of \textit{oriented quadratic forms} over $\R$. 

We define an equivalence relation, denoted by $\sim$, on $\OQF$ by identifying $(f,s)$ with $(-f,-s)$ for all $(f,s)\in \OQF$. For each $(f,s)$ let 
\[
[f,s] \sceq \{ (f,s) , (-f,-s) \}
\]
denote its equivalence class, which we call a \textit{projective oriented quadratic form} over $\R$. Further, let
\[
 \POQF \sceq \POQF(\R) \sceq \OQF(\R)/_{\sim}
\]
denote the set of all equivalence classes, that is the set of projective oriented quadratic forms. 

For $[f,s]\in \POQF$ we call
\[
 x_a\big([f,s]\big) \sceq x_a\big( (f,s) \big) \quad\text{resp.}\quad x_r\big( [f,s] \big) \sceq x_r\big( (f,s) \big)
\]
the \textit{attracting} resp.\@ \textit{repelling zero} of $[f,s]$. Since $x_+(f) = x_-(-f)$, these notions are indeed well-defined.

To simplify notations, for $f=[A,B,C]$ and $s=+1$  we denote $[f,s]$ by
\[
 [A,B,C,s] \quad\text{or}\quad [A,B,C,+],
\]
and analogously if $s=-1$.

We define an action of $G$ on $\QF$ by 
\[
 g.f(x,y) = f\left(\left( g^{-1}{x\choose y} \right)^\top\right) = f(dx-by, -cx + ay)
\]
for all $g = \textbmat{a}{b}{c}{d}\in G$, $f\in\QF$. 
If $f=[A,B,C]$ then
\[
 g.f = [ Ad^2 - Bcd +Cc^2, -2Abd + B(ad+bc) - 2Cac, Ab^2 -Bab + Ca^2 ],
\]
which immediately shows well-definedness. The action of $G$ on $\QF$ extends to $\OQF$ by 
\[
 g.(f,s) \sceq (g.f, s)
\]
for each $g\in G$, $(f,s)\in \OQF$. Moreover, it then decends to $\POQF$ by 
\[
 g.[f,s] \sceq [g.f, s]
\]
for all $g\in G$, $[f,s]\in\POQF$.  Then the map
\begin{equation}\label{phi2}
 \mapzwei\colon \POQF \to \FP,\quad [f,s] \mapsto \big( x_a([f,s]), x_r([f,s]) \big),
\end{equation}
is $G$-equivariant.

\begin{thm}\label{mqfm}
The map 
\[
\mapvier\colon \POQF\to G_h,\ [A, B, C, s] \mapsto \bmat{\frac{-B+s\sqrt{D+4}}{2}}{-C}{A}{\frac{B+s\sqrt{D+4}}{2}}
\]
is a $G$-equivariant bijection with
\[
 \mapeins\circ\mapvier = \mapzwei.
\]
Its inverse is given by
\[
 \mapvier^{-1}\colon G_h \to \POQF,\quad \bmat{a}{b}{c}{d} \mapsto [c, d-a, -b, \sgn(a+d)].
\]
\end{thm}

\begin{proof}
One immediately checks that these maps are well-defined and bijective to each other. The remaining statements are then proven by straightforward calculations.
\end{proof}

\subsection{Summary of maps}\label{summaps}

We have the commutative diagram 
\[
\xymatrix{
\POQF \ar[rr]^{\mapvier} \ar[dd]_{\mapzwei} && G_h  \ar[ddll]_{\mapeins} \ar[dd]^{\mapdrei}
\\\\
 \FP && \FP\times\R_{>0} \ar[ll]_{\pr_1}
\\
\UGeo/_\sim \ar[u]^\mgf && \UGeo/_\sim\times\R_{>0}  \ar[u]_{\mgf\times\{\id\}},
}
\]
where all maps are $G$-equivariant,  $\mapdrei$, $\mapvier$, $\mgf$ and $\mgf\times\{\id\}$ are bijections, and all maps are surjective.

\section{Quadratic forms associated to Hecke triangle groups, periodic geodesics, and identifications}\label{sec:rel_hecke}

Recall the map $\mapvier\colon \POQF\to G_h$ from Theorem~\ref{mqfm} and let $\Gamma = \Gamma_q$ be a Hecke triangle group. Let
\[
\Gamma_h \sceq \Gamma\cap G_h
\]
denote the set of its hyperbolic elements. Then 
\[
 \POQF(\Gamma_h) \sceq \mapvier^{-1}(\Gamma_h)
\]
is the set of \textit{projective oriented quadratic forms associated to $\Gamma$}, for short, \textit{$\Gamma$-forms}. 

The proof of the reduction theory relies crucially on the relations between periodic geodesics on $X=\Gamma\backslash\h$, $\Gamma$-forms and the $\Gamma$-equivalence classes of hyperbolic matrices in $\Gamma$. In this section we specialize and extend the identifications from Section~\ref{sec:relations} accordingly. 

The operations in the algorithm of the reduction theory are performed on the pairs of zeros of $\Gamma$-forms respectively on the endpoints of lifts of periodic geodesics on $X$ to $\h$ (see Section~\ref{sec:redtheory} below). Even when restricted to $\Gamma_h$ respectively to $\Gamma$-forms, the maps $\mapeins\colon G_h\to\FP$ (see \eqref{phi1}) and $\mapzwei\colon \POQF\to\FP$ (see \eqref{phi2}) are not bijections since positive powers of any hyperbolic element have identical fixed points. To turn these into $\Gamma$-equivariant bijections we use a natural notion of level/multiplicity for hyperbolic elements, $\Gamma$-forms and periodic geodesics.

For each hyperbolic element $g\in \Gamma_h$ there exists a maximal integer $n=n(g)\in\N$ such that 
\[
 g = h^n
\]
for some (necessarily hyperbolic) $h\in\Gamma$. The element $h$ is unique. We call $n(g)$ the \textit{level} of $g$. If $n(g) = 1$ then we call $g$ \textit{primitive}. Let $\Gamma_p$ denote the set of primitive hyperbolic elements in $\Gamma$. Further, for $n\in\N$, let
\begin{align*}
\Gamma_{h,n} &\sceq \{ g\in\Gamma_h \mid n(g) = n \} = \{ h^n \mid h \in \Gamma_{p}\}.
\end{align*}
We call the elements of
\[
 \POQF(\Gamma_{h,n}) \sceq \mapvier^{-1}(\Gamma_{h,n})
\]
\textit{$\Gamma$-forms of level $n$}.

Let 
\[
 \FP(\Gamma_h) \sceq \{ (w_a(g), w_r(g) ) \mid g\in\Gamma_h\}
\]
denote the set of fixed point pairs of the elements in $\Gamma_h$. Recall from Proposition~\ref{displacement} that $t_g$ denotes the displacement length of $g\in \Gamma_h$. Proposition~\ref{displacement}\eqref{displaceindep} immediately implies that 
\[
 t_{g^n} = n t_g
\]
for all $n\in\N$. For $n\in\N$ let
\[
 \FP(\Gamma_{h,n}) \sceq \{(w_a(g), w_r(g), t_g) \mid g\in\Gamma_{h,n}\}.
\]

\begin{lemma}
Let $n\in\N$. 
\begin{enumerate}[{\rm (i)}]
\item The $G$-actions on $G$, $\POQF$ and $\FP$ induce $\Gamma$-actions on $\Gamma_{h,n}$, $\POQF(\Gamma_{h,n})$, $\FP(\Gamma_{h,n})$ and $\FP(\Gamma_h)$. 
\item The restricted maps
\begin{align*}
\mapdrein & \sceq \mapdrei\vert_{\Gamma_{h,n}}\colon \Gamma_{h,n} \to \FP(\Gamma_{h,n}),
\\
\mapviern & \sceq \mapvier\vert_{\POQF(\Gamma_{h,n})}\colon \POQF(\Gamma_{h,n}) \to \Gamma_{h,n},
\\
\mapeinsn & \sceq \mapeins\vert_{\Gamma_{h,n}}\colon \Gamma_{h,n} \to \FP(\Gamma_h),
\\
\mapzwein & \sceq \mapzwei\vert_{\POQF(\Gamma_{h,n})}\colon \POQF(\Gamma_{h,n}) \to \FP(\Gamma_h),
\intertext{and}
\pr_{1,n}&\colon \FP(\Gamma_{h,n}) \to \FP(\Gamma_h)
\end{align*}
are $\Gamma$-equivariant bijections.
\end{enumerate}
\end{lemma}

\begin{proof}
Theorems~\ref{mqfm} and \ref{mmg} and the definitions of $\Gamma_{h,n}, \POQF(\Gamma_{h,n})$ and $\FP(\Gamma_{h,n})$ imply immediately the bijectivity of $\mapdrein$ and $\mapviern$. Further, the level of hyperbolic elements in $\Gamma$ is stable under conjugation with elements in $\Gamma$. This fact and the $G$-equivariance of $\mapdrei,\mapvier,\mapeins,\mapzwei,\pr_1$ (see Section~\ref{summaps}) yield the claimed $\Gamma$-actions (in particular, their well-definedness) and the $\Gamma$-equivariance of the maps $\mapdrein, \mapviern, \mapeinsn, \mapzwein$ and $\pr_{1,n}$. 

It remains to show that $\mapeinsn$ is bijective. From this, bijectivity of $\mapzwein$ and $\pr_{1,n}$ follows immediately. Let $g_1,g_2\in\Gamma_{h,n}$ such that 
\[
 \mapeins(g_1) = \mapeins(g_2).
\]
Let $(x,y)\sceq\mapeins(g_1)$ and consider 
\[
 \mc S\sceq \{ h\in\Gamma \mid h.x = x,\ h.y = y \}
\]
Then $\mc S$ obviously contains $g_1$ and $g_2$. Since $\Gamma$ is discrete, $\mc S$ is cyclic. Thus, 
\[
\mc S = \langle h_0\rangle = \langle h_0^{-1}\rangle 
\]
for some $h_0\in\Gamma$, which is necessarily primitive hyperbolic. Thus, without loss of generality,
\[
 g_1 = h_0^n = g_2.
\]
This shows the injectivity of $\mapeinsn$; the surjectivity is obvious.
\end{proof}

%

Recall from Section~\ref{sec:hecketriangle} that $\PGeo(\Gamma)$ denotes the set of periodic geodesics on $X$. For $\wh\gamma\in\PGeo(\Gamma)$ let $t_{\wh\gamma}$ denote the minimal period length of $\wh\gamma$. For $n\in\N$ let 
\[
 \PGeo(\Gamma)_n \sceq \{ (\wh\gamma, nt_{\wh\gamma}) \mid \wh\gamma\in\PGeo(\Gamma)\}
\]
denote the set of periodic geodesics on $X=\Gamma\backslash\h$ where the period length is considered with multiplicity $n$. Recall that we consider geodesics on $\h$ to be equivalent if they coincide after a reparametrization. We use the same notion of equivalence for geodesics on $X$ as well. We denote by
\[
 \PGeo(\Gamma)/_\sim 
\]
the set of equivalence classes of geodesics in $\PGeo(\Gamma)$, and by $[\wh\gamma]$  the equivalence class of $\wh\gamma\in\PGeo(\Gamma)$. Analogously, we let 
\[
\PGeo(\Gamma)_n/_\sim
\]
be the set of equivalence classes of geodesics in $\PGeo(\Gamma)_n$, and we use $[(\wh\gamma, nt_{\wh\gamma})]$ to denote the equivalence class of $(\wh\gamma, nt_{\wh\gamma})\in \PGeo(\Gamma)_n$.

Obviously, for each $n\in\N$, 
\begin{equation}\label{leveliso}
 \PGeo(\Gamma) \cong \PGeo(\Gamma)_n \quad\text{and}\quad  \PGeo(\Gamma)/_\sim \cong \PGeo(\Gamma)_n/_\sim.
\end{equation}
Further, the map
\begin{equation}\label{specialproj}
 \alpha\colon\FP(\Gamma_h) \to \PGeo(\Gamma)/_\sim,\quad \big( \gamma(+\infty), \gamma(-\infty) \big) \to [\wh\gamma],
\end{equation}
where $\gamma$ is any geodesic on $\h$ which represents $\wh\gamma$, is constant on $\Gamma$-orbits in $\FP(\Gamma_h)$.

The following proposition is well-known.

\begin{prop}\label{allisos}
Let $n\in\N$. Then $\PGeo(\Gamma)_n/_\sim$, the set of $\Gamma$-equivalence classes in $\Gamma_{h,n}$, and the set of $\Gamma$-equivalence classes of $\POQF(\Gamma_{h,n})$ are bijective. The bijections are induced by the maps
\[
 \pr_{1,n}\circ\mapviern\colon \POQF(\Gamma_{h,n}) \to \FP(\Gamma_h) \quad\text{and}\quad  \pr_{1,n}\circ\mapdrein\colon \Gamma_{h,n} \to \FP(\Gamma_h),
\]
any section of the map $\alpha$ in \eqref{specialproj}, and the natural bijections in \eqref{leveliso}.
\end{prop}

\section{The reduction theory}\label{sec:redtheory}

In this section we propose a reduction theory for $\Gamma$-forms. To that end we define a notion of reducedness such that each $\Gamma$-equivalence class of $\Gamma$-forms contains finitely many and at least one reduced form. We present an algorithm---the reduction algorithm---which, for any given $\Gamma$-form, finds all equivalent reduced forms. More precisely, depending on the given $\Gamma$-form, the algorithm might first produce finitely many non-reduced equivalent forms, the `preperiod'. As soon as it has found one reduced equivalent form it continues to produce other reduced equivalent forms until it reproduces this first reduced form. In this way, the equivalent reduced forms are ordered in a cycle, the `period'. Finally, we show that there are only finitely many reduced $\Gamma$-forms for any fixed discriminant.

We call $F\in \POQF(\Gamma_h)$ \textit{reduced} if 
\begin{equation}\label{defiformred}
 x_a(F) > 0 > x_r(F).
\end{equation}

Recall from \eqref{generators} the elements $T,S\in\Gamma$. Let $U\sceq TS$ and, for $k\in\Z$, $g_k \sceq (U^k S)^{-1}$. A straightforward calculation shows 
\begin{equation}\label{formulargk}
 g_k^{-1} = \frac{1}{\sin\frac{\pi}{q}} \bmat{ \sin\left(\frac{k}{q}\pi\right) }{ \sin\left(\frac{k+1}{q}\pi\right) }{ \sin\left(\frac{k-1}{q}\pi\right) }{ \sin\left(\frac{k}{q}\pi\right) }.
\end{equation}
Further, $U^q = \id$ and $g_{k+q} = g_k$. Moreover, for $k=1,\ldots, q-2$, we have
\[
 0 = g_{q-1}^{-1}.0,\quad g_{k+1}^{-1}.\infty = g_k^{-1}.0, \quad g_1^{-1}.\infty=\infty. 
\]
Thus, the intervals $(g_k^{-1}.0, g_k^{-1}.\infty)$, $k=1,\ldots, q-1$, are almost a partition of $(0,\infty)$. More precisely, their union is disjoint and equals 
\[
 D_0\sceq (0,\infty) \setminus \left\{g_k^{-1}.\infty \ \left\vert\ k=1,\ldots, q-1 \vphantom{g_k^{-1}}\right.\right\} = \bigcup_{k=1}^{q-1} \left(g_k^{-1}.0, g_k^{-1}.\infty\right).
\]
For $x\in D_0$ and $k\in\{1,\ldots, q-1\}$ we define
\[
 \interval(x) \sceq k
\]
if and only if $x\in (g_k^{-1}.0, g_k^{-1}.\infty)$.

\begin{algorithm}
\caption{The reduction algorithm}\label{redalgo}
\begin{algorithmic}[1]
\Statex \textbf{Input:}  $F\in \POQF(\Gamma_h)$
\Statex

\Statex\Comment{Initialization}
\State $F_0\gets F$ 
\State $(x_0,y_0) \gets \mapzwei(F_0)$
\State $h_0 \gets \id$
\State $j \gets 0$ 
\Statex

\Statex \Comment{Calculation of preperiod}
\If{$x_0<0$} \Comment{Enforce $x_j>0$}
\State $F_1 \gets S.F_0$
\State $(x_1,y_1) \gets (S.x_0,S.y_0)$
\State $h_1 \gets S$
\State $j\gets 1$
\EndIf
\While{$F_j$ not reduced} \Comment{Note that $x_j>0$ throughout}
\State $k\gets \interval(x_j)$
\State $F_{j+1} \gets g_k.F_j$
\State $(x_{j+1},y_{j+1}) \gets  g_k.(x_j,y_j)$
\State $h_{j+1} \gets h_jg_k^{-1}$
\State $j \gets j+1$
\EndWhile
\State $m\gets j$ \Comment{Store length of preperiod}
\Statex

\Statex\Comment{Calculation of period}
\State $a_0 \gets \id$ 
\Repeat
\State $k\gets \interval(x_j)$
\State $F_{j+1} \gets g_k.F_j$
\State  $(x_{j+1},y_{j+1}) \gets g_k.(x_j,y_j)$
\State $j \gets j+1$
\State $a_{j-m} \gets  a_{j-m-1}g_k^{-1}$ \Comment{Note that $j$ is already enlarged by $1$}
\Until{$(x_{j},y_{j})=(x_m,y_m)$} 
\State $p\gets j-m$
\Statex

\Statex\Comment{Calculation of level and hyperbolic element}
\State $n \gets 1$
\While{$\mapvier^{-1}(h_ma_p^n h_m^{-1})\not=F$}
\State $n\gets n+1$
\EndWhile
\Statex

\Statex\Comment{Output}
\State \textbf{return} Level of $F$ is $n$.
\State \textbf{return} Preperiod is $F_0, F_1, \ldots, F_m$.
\State \textbf{return} Period is $F_{m+1}, F_{m+2}, \ldots, F_{m+p}$.
\State \textbf{return} Associated hyperbolic element is $h_ma_p^n h_m^{-1}$.
\end{algorithmic}
\end{algorithm}

The proof that Algorithm~\ref{redalgo} below indeed constitutes a reduction theory (Theorem~\ref{redprimitive} below) takes advantage of the discretization and symbolic dynamics for the geodesic flow on $X$ as provided in \cite{Pohl_Symdyn2d}. We recall some results from \cite{Pohl_Symdyn2d} and extend these for our purposes.

\subsubsection*{Reduced geodesics}
We start by characterizing the notion of reducedness of $\Gamma$-forms in terms of geodesics and subsets in $S\h$. In analogy to \eqref{defiformred} we call a geodesic $\gamma$ on $\h$ or $[\gamma]\in \UGeo/_\sim$ \textit{reduced} if 
\[
\gamma(+\infty)>0>\gamma(-\infty).
\]
We say that a geodesic $\gamma$ on $\h$ \textit{intersects}  $M\subseteq S\h$ if $\gamma'(t)\in M$ for some $t\in\R$. Let 
\[
 C' \sceq \left\{  a\tfrac{\partial}{\partial x}\vert_{iy} + b\tfrac{\partial}{\partial y}\vert_{iy}   \left\vert\ a>0,\ b\in\R,\ y>0\vphantom{ a\tfrac{\partial}{\partial x}\vert_{iy} + b\tfrac{\partial}{\partial y}\vert_{iy}} \right.\right\}
\]
be the set of unit tangent vectors in $S\h$ that are based on $i\R_{>0}$ and point `to the right' (see Figure~\ref{funddomSHpic}). The following characterization is obvious.

\begin{lemma}
\begin{enumerate}[{\rm (i)}]
\item A geodesic $\gamma$ on $\h$ is reduced if and only if $\gamma$ intersects $C'$.
\item Let $F\in\POQF(\Gamma_h)$ and $[\gamma]\in\UGeo/_\sim$ such that $\mgf([\gamma]) = \mapzwei(F)$. Then $F$ is reduced if and only if $\gamma$ is reduced.
\end{enumerate}
\end{lemma}

\begin{center}
\begin{figure}[h]
\includegraphics{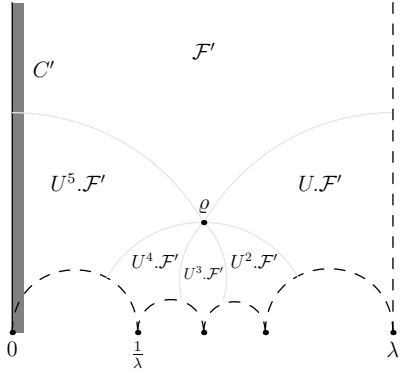} 
\caption{The base set for the fundamental set $\wt{\mc F}$ for $\Gamma$ in $S\h$ for $q=6$. At each point in the interior and at the geodesic arc $i\R_{>0}$ there is based at least one unit tangent vector. The set $C'$ in $S\h$ is indicated in dark gray.}\label{funddomSHpic}
\end{figure}
\end{center}

\subsubsection*{The set $C'$ and a cross section for the geodesic flow}

In \cite{Pohl_Symdyn2d}, the set $C'$ is constructed as a set of representatives for a cross section for the geodesic flow on $X$. It is crucial for developing a certain discrete dynamical system which is closely related to the first return map of this cross section and which is lurking in the background of the proof of Theorem~\ref{redprimitive} below.

\begin{defi}
Let $\wh M$ be a subset of $SX$ and $\wh\gamma$ a geodesic on $X$. 
\begin{enumerate}[{\rm (i)}]
\item\label{cross1} We say that $\wh\gamma$ \textit{intersects} $\wh M$ if $\wh\gamma'(t_0)\in\wh M$ for some $t_0\in \R$. In this case, we call $\wh\gamma'(t_0)$ an \textit{intersection point} and $t_0$ an \textit{intersection time}. Moreover, we say that the intersection is \textit{discrete} (in space and time) if there exists $\eps>0$ such that
\[
 \wh\gamma'((t_0-\eps,t_0+\eps)) \cap \wh U = \{ \wh\gamma'(t_0) \}.
\]
\item\label{cross2} Suppose that $\wh\gamma$ intersects $\wh M$ discretely in $\wh\gamma'(t_0)$. If  
\[
 t_1 = \min\{ t>t_0 \mid \wh\gamma'(t)\in \wh U\}
\]
exists then we call $(\wh\gamma'(t_1), t_1)$ (or also only $\wh\gamma'(t_1)$) the \textit{next intersection (point and time)} relative to $\wh\gamma'(t_0)$. If all intersections between $\wh\gamma$ and $\wh M$ are discrete then we call $(\wh\gamma'(t_j))_{j\in\Z}$ its \textit{sequence of intersections} if, for each $j\in\Z$, $\wh\gamma'(t_j)\in\wh U$ and $\wh\gamma'(t_{j+1})$ is the next relative intersection.
\item We call $\wh M$ a \textit{cross section} for the geodesic flow on $X$ if each intersection between any geodesic on $X$ and $\wh M$ is discrete, and each periodic geodesic intersects $\wh M$ (infinitely often).
\item We say that a subset $M'$ of $S\h$ is a \textit{set of representatives} for $\wh M$ if the quotient map $\pi\colon S\h\to SX$ induces a bijection between $M'$ and $\wh M$.
\item Suppose that $M$ is a subset of $S\h$ and $\gamma$ a geodesic on $\h$. We use the notions in \eqref{cross1} and \eqref{cross2} \textit{mutatis mutandis} for the intersections of $\gamma$ and $M$.
\end{enumerate}
\end{defi}

A fundamental set for the action of $\Gamma$ on $S\h$ is a subset of $S\h$ which contains exactly one representative of each $\Gamma$-orbit. The set $C'$ is essentially a component of the `visual boundary' of a fundamental set $\wt{\mc F}$ for the action of $\Gamma$ on $S\h$ which is constructed from the fundamental domain $\mc F$ in \eqref{funddom} for the action of $\Gamma$ on $\h$ as described in the following: We shift over the left half of $\mc F$ with the side-pairing translation $T$ and slightly adapt the boundaries to find the fundamental domain 
\[
 \mc F'\sceq \{ z\in\h\mid \lambda > \Rea z > 0,\ |z|>1,\ |z-\lambda|>1 \}
\]
for the action of $\Gamma$ on $\h$ (see Figure~\ref{funddom2}). 
\begin{center}
\begin{figure}[h]
\includegraphics{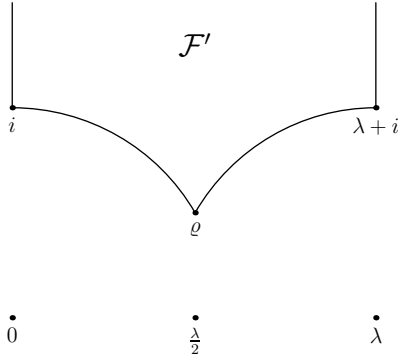} 
\caption{Another fundamental domain for $\Gamma$ in $\h$.}\label{funddom2}
\end{figure}
\end{center}
The side-pairings of $\mc F'$ are given by $T$ which pairs the vertical sides
$[i,\infty)_\gamma$ and $[\lambda+i,\infty)_\gamma$, and by $U$ which pairs the bottom sides $[i,\varrho]_\gamma$ and $[\varrho,\lambda+i]_\gamma$. The stabilizer of $\varrho$ is
\[
 \Stab_\Gamma(\varrho) = \left\langle U \right\rangle = \{ U, U^2,\ldots, U^q=\id\}.
\]
For any subset $M$ of $\h$ and unit tangent vector $v\in S\h$ we say that 
\begin{enumerate}[{\rm (i)}]
\item $v$ \textit{points into} $M$ if $\gamma_v( (0,\eps) ) \subseteq M^\circ$ for some $\eps >0$,
\item $v$ \textit{points out of} $M$ if $\gamma_v( (0,\eps) ) \subseteq (\h\setminus M)^\circ$ for some $\eps>0$, and
\item $v$ \textit{points along} the boundary of $M$ if $\gamma_v( (0,\eps) ) \subseteq \partial M$ for some $\eps>0$.  
\end{enumerate}
Note that in any of these cases, the base point of $v$ is contained in $\overline M$. Let 
\[
 \base\colon S\h \to \h
\]
denote the projection on base points.

A fundamental set for the action of $\Gamma$ on $S\h$ is given by the set $\wt{\mc F}'$ which consists of all unit tangent vectors with base point in $\overline{\mc F}'$ that point into $\mc F'$ and a certain choice of those unit tangent vectors which point along $\partial\mc F'$. We partition $\wt{\mc F}'$ into $q$ subsets as follows:

Let 
\[
 D\sceq \{ z\in \overline{\mc F}' \mid \lambda>\Rea z>0\}
\]
denote the closure of the fundamental domain $\mc F'$ in $\h$ but without the two vertical boundary sides, and let
\[
 \wt{\mc F}'' \sceq \{ v\in \wt{\mc F}' \mid \base(v)\in D\}
\]
denote the subset of the fundamental set $\wt{\mc F}'$ in $S\h$ which consists of those unit tangent vectors that are not based on any of the two vertical sides of $\mc F'$. We fix any partition of $\wt{\mc F}''$ into $q$ (pairwise disjoint) subsets $\wt W_j'$, $j=0,\ldots, q-1$, of $\wt{\mc F}''$ such that $\base(\wt W_j') = D$ for all $j\in\{0,\ldots, q-1\}$. We set
\begin{align*}
\wt W_0 & \sceq \{ v\in \wt{\mc F}' \mid \base(v) \in [i, \infty)_\gamma \} \cup \wt W_0',
\\
\wt W_1 & \sceq \{ v\in \wt{\mc F}' \mid \base(v) \in [\lambda+i, \infty)_\gamma \} \cup \wt W_1',
\intertext{and}
\wt W_j & \sceq \wt W_j'\quad\text{for $j=2,\ldots, q-1$.}
\end{align*}
Then
\begin{equation}\label{funddomSH}
 \wt{\mc F} \sceq \bigcup_{j=0}^{q-1} U^{-j}\wt W_j
\end{equation}
is a fundamental set for $\Gamma$ in $S\h$, and $C'$ is the set of elements in $\wt{\mc F}$ which are based on $i\R_{>0}$, possibly up to some vectors which point along $i\R_{>0}$ and which we may ignore for our purposes (see Figure~\ref{funddomSHpic}). Let 
\[
 B \sceq \base(\wt{\mc F}).
\]
By analyzing the previously described construction of $\wt{\mc F}$ we find the following properties of $B$ (cf.\@ \cite{Pohl_Symdyn2d}):
\begin{enumerate}[(A)]
\item\label{propBA}  The set $B$ is a geodesically convex polyhedron of finite area. All boundary points of $B$ in $P^1(\R)$ are cuspidal.
\item\label{propBB} The boundary $\partial B$ of $B$ in $\h$ decomposes into $q$ geodesic arcs, namely the arcs $(g_k^{-1}.0,g_k^{-1}.\infty)_\gamma$ for $k=0,\ldots, q-1$.
\item\label{propBC} We have $\base(C) = \Gamma.\partial B$. The intersection between $\base(C)$ and $B^\circ$ is empty.
\item\label{propBD} For $k=0,\ldots, q-1$ let 
\[
 C_k \sceq \left\{ v\in C\ \left\vert\ \base(v)\in g_k^{-1}.(i\R_{>0}),\ \gamma_v( (0,\infty) ) \cap B = \emptyset \right.\right\}
\]
denote the set of unit tangent vectors in $C$ that are based on the geodesic arc $g_k^{-1}.(i\R_{>0}) = (g_k^{-1}.0,g_k^{-1}.\infty)_\gamma$ and that point out of $B$. Then 
\[
 C_k = g_k^{-1}.C',
\]
and $g_k$ is the unique element in $\Gamma$ with this property.
\end{enumerate}

Let 
\[
 \wh C \sceq  \pi(C'). 
\]

\begin{prop}[\cite{Pohl_Symdyn2d}]\label{crosssection}
The set $\wh C$ is a cross section for the geodesic flow on $X$ with $C'$ as set of representatives. 
\end{prop}

We provide a partial proof of Proposition~\ref{crosssection} in Lemma~\ref{lem:intersect} below. For a complete proof we refer to \cite{Pohl_Symdyn2d}.

\subsubsection*{Sequences of intersections and a proof of the reduction theory}

For the proof of the reduction algorithm we need to understand the sequences of intersections between periodic geodesics on $X$ and the cross section $\wh C$, and which effect they have on certain lifts of periodic geodesics to $\h$. For that let
\[
 C \sceq \Gamma.C'. 
\]

The first two statements of the following lemma provide a partial proof of Proposition~\ref{crosssection}.

\begin{lemma}\label{lem:intersect}
Let $\gamma\colon \R\to\h$ be  a lift of a periodic geodesic $\wh\gamma$ on $\Gamma\backslash\h$. 
\begin{enumerate}[{\rm (i)}]
\item\label{intersect1} Then $\gamma$ intersects $C$. In particular, there exists a lift of $\wh\gamma$ which intersects $C'$.
\item\label{intersect2} The geodesic $\gamma$ intersects $C'$ at most once. 
\item\label{intersect3} If $\gamma$ intersects $C'$, then the next intersection is on $g_k^{-1}.C'$ for some $k\in\{1,\ldots, q-1\}$. For $k=1,\ldots,q-1$, it is on $g_k^{-1}.C'$ if and only if $\gamma(\infty) \in (g_k^{-1}.0, g_k^{-1}.\infty)$. 
\end{enumerate}
\end{lemma}

\begin{proof}
Pick $t_0\in\R$ and consider $v\sceq \gamma'(t_0)$. Since $\wt{\mc F}$ is a fundamental set for the action of $\Gamma$ on $S\h$, we find (a unique) $h\in\Gamma$ such that $w\sceq h.v\in\wt{\mc F}$. Let $\eta \sceq h.\gamma$. 

Since $\gamma$, and hence $\eta$, is a lift of a periodic geodesic, neither of its limit points is cuspidal. By \eqref{propBA} and \eqref{propBB}, the boundary of $B$ consists of geodesic arcs, each of which connects two cuspidal points. Hence $\eta(\R)$ is not contained in the boundary of $B$, and $w$ does not point along the boundary of $B$. In fact, $w$ points into $B$. Since $B$ is geodesically convex and of finite area, and $\eta(\infty)$ is not cuspidal, the set $\eta((t_0,\infty))$ intersects the boundary of $B$ in exactly one point, say in $\eta(t_1)$.

By \eqref{propBD}, $\eta(t_1)$ is contained in $g_k^{-1}.(i\R_{>0})$ for some $k\in\{0,\ldots, q-1\}$. The tangent vector $u\sceq \eta'(t_1)$ points out of $B$, hence is contained in $g_k^{-1}.C$. Thus, $g_k.u \in C'$, the geodesic $g_kh.\gamma$ is a lift of $[\gamma]$ that intersects $C'$, and $\gamma$ intersects $h^{-1}g_k^{-1}.C'\subseteq\wt C$. This proves \eqref{intersect1}.

Statement~\eqref{intersect2} follows from geodesic convexity of $\gamma(\R)$ and $\base(C') = i\R_{>0}$.

For \eqref{intersect3} suppose that $\gamma'(t_0) \in C'$. Hence, in the argument above, $h=\id$, $w=v$, $\eta=\gamma$, and the next intersection is $\gamma'(t_1) \in g_k^{-1}.C'$ for some $k\in\{0,\ldots, q-1\}$. By geodesic convexity, $k=0$ is impossible, and $\gamma'(t_1)\in g_k^{-1}.C'$ if and only if $\gamma(\infty)\in (g_k^{-1}.0,g_k^{-1}.\infty)$.
\end{proof}

The standard geodesic $\gamma_\st$ defines two half spaces in $\h$, the \textit{left half space}
\[
 L(\gamma_\st) \sceq \{ z\in\h \mid \Rea z < 0\}
\]
and the \textit{right half space}
\[
 R(\gamma_\st) \sceq \{ z\in\h \mid \Rea z > 0\}.
\]
Let $g\in G$. Then we call
\[
 L(g.\gamma_\st) \sceq g.L(\gamma_\st)
\]
the \textit{left half space} of $g.\gamma_\st$, and 
\[
 R(g.\gamma_\st) \sceq g.R(\gamma_\st)
\]
the \textit{right half space} of $g.\gamma_\st$. By Proposition~\ref{twopoint}, this definition yields a left/right half space for any geodesic on $\h$. 

Let $\gamma\in\UGeo$ and $x\in P^1(\R)$. By a slight abuse of notions, we say that $x$ \textit{belongs to} $L(\gamma)$ (notation: $x\overline{\in}L(\gamma)$) if $x$ is contained in the closure of $L(\gamma)$ in $\h\cup P^1(\R)$ but $x\notin \{\gamma(\pm\infty)\}$. Moreover, we use the appropriately modified definition of $x$ \textit{belonging to} $R(\gamma)$.

\begin{lemma}\label{lem:goodelem}
For any pair $(x,y)$ of distinct non-cuspidal points there exists $g\in\Gamma$ such that $x\overline{\in} R(g.\gamma_\st)$ and $y\overline{\in} L(g.\gamma_\st)$. If $x>0$ then we can choose $g\in\Gamma$ such that $\infty > g.0\geq 0$.
\end{lemma}

\begin{proof}
Let $I(x,y) = (x,y)_\gamma$ denote the geodesic arc connecting $x$ and $y$, and note that $x,y$ are both real since they are non-cuspidal. Recall that $B$ is a (closed) geodesically convex polyhedron bounded by maximal geodesic arcs only, and all its vertices are cuspidal points. The points $x$ and $y$ are non-cuspidal. Therefore, if $I(x,y)$ intersects $h.B$ for some $h\in\Gamma$ then $I(x,y)$ intersects a boundary component of $h.B$, and vice versa. 

Recall now that all sides of $B$ are $\Gamma$-translates of $i\R_{>0}$ (the geodesic arc of the standard geodesic $\gamma_\st$). Hence, each side of $B$ is the geodesic arc 
\[
 (a.0, a.\infty)_\gamma = (aS.0,aS.\infty)_\gamma
\]
for some $a\in\Gamma$. Thus, if $I(x,y)$ intersects $h.B$ for some $h\in\Gamma$ then there exists $g=g\in\Gamma$ such that either 
\begin{equation}\label{sit1}
x \overline{\in} R(g.\gamma_\st) \quad\text{and}\quad y\overline{\in} L(g.\gamma_\st)
\end{equation}
or
\begin{equation}\label{sit2}
y \overline{\in} R(g.\gamma_\st) \quad\text{and}\quad x\overline{\in} L(g.\gamma_\st) .
\end{equation}
For the first statement of this lemma, it remains to show that we can always realize \eqref{sit1}.

Suppose that $I(x,y) = \gamma(\R)$ for $\gamma\in \UGeo$. Suppose further, without loss of generality, that $\infty\overline{\in} R(\gamma)$.  Recall that the $\Gamma$-translates of $B$ tile the upper half plane $\h$.  Let 
\[
 \wt B \sceq \bigcup_{g\in\Gamma, gB\subseteq R(\gamma)} gB 
\]
denote the union of the $\Gamma$-translates of $B$ that are completely contained in $R(\gamma)$. Due to the tiling properties of $B$, the set $\wt B$ is a geodesically convex subset all of whose sides in $\h$ are maximal geodesic arcs. Let $J$ be the (unique) side of $\wt B$ which is nearest to $I(x,y)$. This is, $I(x,y)$ is completely contained in one of the half spaces defined by $J$, and $\wt B$ is contained in the other half space. Then the side-pairing properties of $B$ and the fact that $\Gamma.B$ covers $\h$ show the existence of $h\in\Gamma$ such that $h.B$ is not contained in $\wt B$ but $J$ is one of its sides. The maximality of $\wt B$ implies that $h.B$ intersects $I(x,y)$, more precisely, that at least one vertex of $h.B$ belongs to $R(\gamma)$ and at least one (other) vertex of $h.B$ belongs to $L(\gamma)$. Thus, $h.B$ has a side which yields \eqref{sit1}, and also one side which yields \eqref{sit2}.

Suppose now that $x>0$. If $y>0$ then 
\[
\{z\in\h\mid \Rea z \leq 0\} \subseteq \wt B.
\]
Hence the construction of $g$ also implies the second statement of this lemma. If $y<0$ then we can choose $g=\id$. This completes the proof.
\end{proof}

\begin{prop}\label{distrC}
Let $h\in\Gamma$ with $h.0,h.\infty\in \R_{\geq 0}\cup\{\infty\}$.
\begin{enumerate}[{\rm (i)}]
\item\label{Ci} If $h.\infty > h.0 \geq 0$ then $h=g_{k_1}^{-1}\cdots g_{k_\ell}^{-1}$ for unique elements $\ell\in\N_0$ and $k_1,\ldots, k_\ell\in\{1,\ldots, q-1\}$.
\item\label{Cii} If $h.0 > h.\infty \geq 0$ then $h=g_{k_1}^{-1}\cdots g_{k_\ell}^{-1}S$ for unique elements $\ell\in\N_0$ and $k_1,\ldots, k_\ell\in\{1,\ldots, q-1\}$.
\item If $h.0\notin\{0,\infty\}$ then $g_{k_1}$ is determined (in cases \eqref{Ci} and \eqref{Cii}) by $h.0\in [g_{k_1}^{-1}.0, g_{k_1}^{-1}.\infty)$. 
\item If $h.0=0$ (which can only happen in case \eqref{Ci}) then $\ell=0$ if $h.\infty=\infty$, and $g_{k_1} = g_{q-1}$ otherwise. 
\item\label{Cv} If $h.0=\infty$ (which can only happen in case \eqref{Cii}) then $\ell=0$ if $h.\infty=0$, and $g_{k_1}^{-1}=T$ otherwise. 
\end{enumerate}
\end{prop}

\begin{proof}
Since $G$ acts two-point simple transitively on $\h$, or equivalently, $G$ acts simple transitively on the space of geodesics on $\h$, the element $h\in\Gamma$ is uniquely determined by the ordered pair $(h.0, h.\infty)$. By \eqref{propBB}, the geodesic arc 
\[
I = (h.0,h.\infty)_\gamma = h.(i\R_{>0})
\]
is a side of $h.B$. Taking advantage of our precise knowledge of the structure of the sides of $B$, a straightforward induction using the side-pairing properties proves this proposition, see Figures~\ref{next1}-\ref{next2}.
\end{proof}

\begin{center}
\begin{figure}[h]
\includegraphics{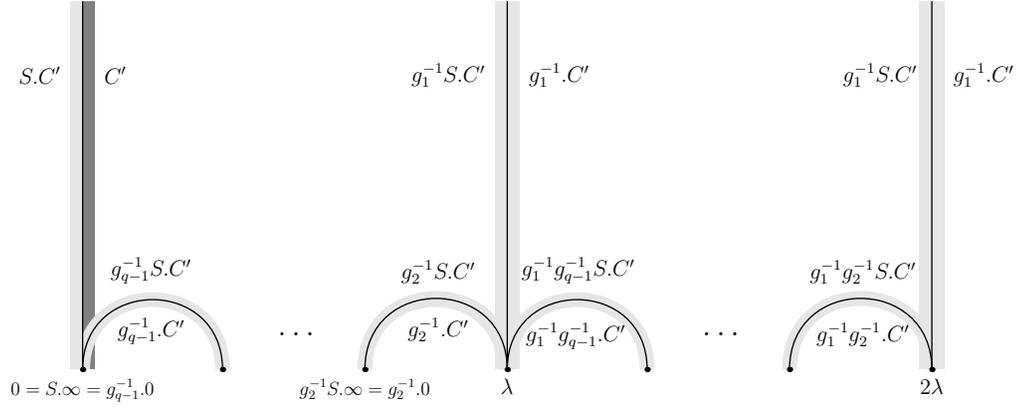} 
\caption{Examples for $\Gamma$-translates of $C'$.}\label{next1}
\end{figure}
\end{center}
\begin{center}
\begin{figure}[h]
\includegraphics{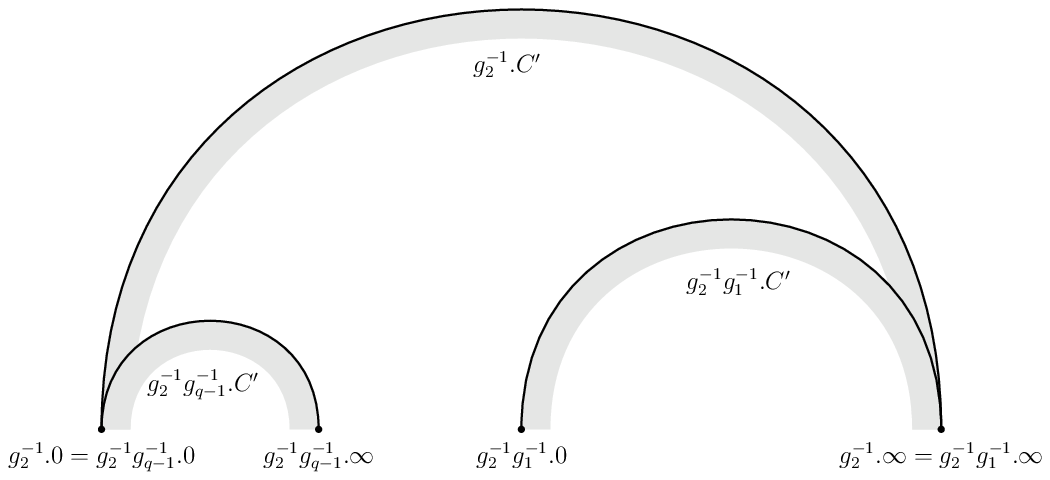} 
\caption{More examples for $\Gamma$-translates of $C'$.}\label{next2}
\end{figure}
\end{center}

\begin{cor}\label{cor:preperiod}
Let $\wh\gamma\in \PGeo(\Gamma)$ and suppose that $\gamma_0$ is a lift of $\wh\gamma$ with $\gamma_0(\infty)>0$. Iterativately for $i=1,2,3,\ldots$ set 
\[
\gamma_i\sceq g_{k_i}.\gamma_{i-1}
\]
where $k_i$ is the unique element in $\{1,\ldots, q-1\}$ such that $\gamma_{i-1}(\infty)\in (g_{k_i}^{-1}.0, g_{k_i}^{-1}.\infty)$.
Then there is $n\in\N_0$ such that $\gamma_n$ is reduced. 
\end{cor}

\begin{proof}
Lemma~\ref{lem:goodelem} yields $g\in\Gamma$ such that 
\[
 0\leq g^{-1}.0 < \gamma_0(\infty) < g^{-1}.\infty \quad\text{and}\quad \gamma_0(-\infty)\notin (g^{-1}.0, g^{-1}.\infty)
\]
By Proposition~\ref{distrC},
\[
 g^{-1} = g_{k_1}^{-1}\cdots g_{k_n}^{-1}
\]
for unique elements $n\in\N_0$, $k_1,\ldots, k_n\in\{1,\ldots, q-1\}$. Each of the $g_{k_j}^{-1}$ acts contracting on $(0,\infty)$. Thus,
\[
 \gamma_0(\infty) \in \big(g_{k_1}^{-1}\cdots g_{k_n}^{-1}.0, g_{k_1}^{-1}\cdots g_{k_n}^{-1}.\infty \big) \subseteq \big(g_{k_1}^{-1}.0, g_{k_1}^{-1}.\infty \big).
\]
By induction we see that $g_{k_j}$ is just the $j$-th acting element in the algorithm in the statement of this corollary.  Thus, $\gamma_n = g.\gamma_0$ and 
\[
 \gamma_n(\infty) > 0 > \gamma_n(-\infty).
\]
Hence, $\gamma_n$ is reduced.
\end{proof}

We call a geodesic $\eta$ on $\h$  a \textit{reduced lift} of a geodesic $\wh\gamma$ on $X$ if $\eta$ is reduced and a lift of $\wh\gamma$. Recall from Sections~\ref{sec:relations}-\ref{sec:rel_hecke} that we consider two geodesics on $\h$ or on $X$ as equivalent if they coincide after a reparametrization. If $[\wh\gamma]\in \PGeo(\Gamma)/_\sim$ and $\gamma\in\UGeo$ is a reduced lift of $\wh\gamma$ then we call $[\gamma]\in\UGeo/_\sim$ a \textit{reduced lift} of $[\wh\gamma]$.

An element $[\wh\gamma]\in \PGeo(\Gamma)/_\sim$ might give rise to more that one reduced lift in $\UGeo/_\sim$. Given any element $[\wh\gamma]\in\PGeo(\Gamma)$, Corollary~\ref{cor:period} below determines all reduced lifts of $[\wh\gamma]$ from any given reduced lift. 
%

\begin{cor}\label{cor:period}
Let $[\wh\gamma]\in \PGeo(\Gamma)/_\sim$ and suppose that $[\gamma_0]\in\UGeo/_\sim$ is one of its reduced lifts. Iteratively for $i=1,2,3,\ldots$ set $[\gamma_i]\sceq [g_{k_i}.\gamma_{i-1}]$ where $k_i$ is the unique element in $\{1,\ldots, q-1\}$ such that $\gamma_{i-1}(\infty)\in (g_{k_i}^{-1}.0, g_{k_i}^{-1}.\infty)$. Then there is $n\in\N$ such that $[\gamma_n]=[\gamma_0]$. If $n$ is chosen to be minimal with this property, then $\mc R\sceq \{[\gamma_0],\ldots, [\gamma_{n-1}]\}$ contains all reduced lifts of $[\wh\gamma]$ and its elements are pairwise distinct. Moreover,
\[
 h=g_{k_1}^{-1}\cdots g_{k_n}^{-1}
\]
is the \textit{primitive} hyperbolic element associated to the pair $\big( \gamma_0(+\infty), \gamma_0(-\infty) \big) \in \FP(\Gamma_h)$ by the isomorphism $\mapeinseins$, and hence 
\[
 \mapeinsn^{-1}\big( \gamma_0(+\infty), \gamma_0(-\infty) \big) = \mapdrein^{-1}\big( (\gamma_0(+\infty), \gamma_0(-\infty), nt_h)\big)= h^n
\]
for all $n\in\N$.
\end{cor}

\begin{proof}
In order to prove that $\mc R=\{[\gamma_0],\ldots, [\gamma_{n-1}]\}$ is the set of reduced lifts for $[\wh\gamma]$ and that its elements are pairwise distinct, we first construct all reduced lifts for $[\wh\gamma]$ and then show that they equal the elements of $\mc R$. 

Fix a representative $\wh\gamma\in\PGeo(\Gamma)$ of $[\wh\gamma]$ and let $(\wh v_j)_{j\in\Z} \in\wh C^\Z$ be the sequence of intersections of $\wh\gamma$ and $\wh C$. Note that $(\wh v_j)_{j\in\Z}$ does not depend on the choice of $\wh\gamma$. As $\wh\gamma$ is a periodic geodesic, this sequence is periodic as well. Let $n\in\N$ be the minimal period length. Hence 
\[
 (\wh v_0, \ldots, \wh v_{n-1})
\]
is a minimal period of $(\wh v_j)_{j\in\Z}$. Choosing a suitable reparametrization of $\wh\gamma$, we may assume that 
\begin{equation}\label{normalization}
 \wh\gamma'(0) = \wh v_0.
\end{equation}
Let 
\[
 t_{n-1}> t_{n-2} > \cdots > t_1 > t_0 = 0
\]
be the corresponding minimal nonnegative intersection times of $\wh\gamma$ with $\wh C$, that is, $\wh\gamma'(t_j) = \wh v_j$ and each $t_j$ is minimal  with this property. Recall that $\pi\vert_{C'}$ is a bijection between $C'$ and $\wh C$, and let  
\[
 v_j \sceq \big(\pi\vert_{C'}\big)^{-1}(\wh v_j),\qquad j=0,\ldots, n-1.
\]
For $j\in\{0,\ldots, n-1\}$ let $\eta_j$ be the geodesic given by
\[
 \eta_j(t) \sceq \gamma_{v_j}(t-t_j).
\]
From
\[
 \eta_j'(t_j) = \gamma_{v_j}'(0) = v_j 
\]
and the normalization \eqref{normalization} of $\wh\gamma$ it follows that $\eta_j$ is a lift of $\wh\gamma$. Since $v_j\in C'$, the geodesic $\eta_j$ is reduced. Moreover, since the elements $v_0,\ldots, v_{n-1}$ are pairwise distinct, Lemma~\ref{lem:intersect}\eqref{intersect2} yields that any two $\eta_j$ are non-equivalent. 

Let $t_p$ be the minimal period length of $\wh\gamma$, and let $\gamma$ be any reduced lift of $\wh\gamma$. Thus $\gamma$ intersects $C'$. Since $\{v_0,\ldots, v_{n-1}\}$ are all possible intersections, we have
\[
 \gamma(t) = \gamma_{v_{j_0}}(t - t_{j_0} + mt_p)  = \eta_{j_0}(t+mt_p)
\]
for some $m\in\Z$ and some $j_0\in\{0,\ldots, n-1\}$. This means that $\gamma$ is equivalent to $\eta_{j_0}$. Thus, the elements $[\eta_0],\ldots,[\eta_{n-1}]$ are pairwise distinct and  $\mc S\sceq \{[\eta_0],\ldots, [\eta_{n-1}]\}$ is the set of reduced lifts of $[\wh\gamma]$.

For the proof that $\mc R=\mc S$  we may assume that 
\[
\gamma_0 = \gamma_{v_0} = \eta_0.
\] 
Since the next intersection of $\wh\gamma$ with $\wh C$ is $\wh v_1 = \wh\gamma'(t_1)$, the next intersection of $\gamma_0$ with $C$ is $\gamma_0'(t_1)$. Lemma~\ref{lem:intersect}\eqref{intersect3} shows that 
\[
\gamma'_0(t_1) = g_{k_1}^{-1}.v_1 \in g_{k_1}^{-1}.C'
\]
for the unique $k_1\in\{1,\ldots, q-1\}$ such that $\gamma_0(\infty)\in (g_{k_1}^{-1}.0, g_{k_1}^{-1}.\infty)$. Further, $\gamma_1 = g_{k_1}.\gamma_0$ is a reduced lift of $\wh\gamma$ with $\gamma_1'(t_1) = v_1$. Thus, $\gamma_1 = \eta_1$. Repeating this argument iteratively for $j=1,\ldots, n-2$ yields $\gamma_2=\eta_2, \ldots, \gamma_{n-1}=\eta_{n-1}$.

For the proof of the statements on the element $h$, we observe that 
\begin{align*}
 \gamma_0'(0) &= v_0,\quad \gamma_0'(t_1) = g_{k_1}^{-1}.v_1, \quad \gamma_0'(t_2) = g_{k_1}^{-1}g_{k_2}^{-1}.v_2, \quad \ldots,
\\
 \gamma_0'(t_{n-1}) &= g_{k_1}^{-1}\cdots g_{k_{n-1}}^{-1}.v_{n-1},
\end{align*}
and finally (recall that $t_p$ is the minimal period length of $\wh\gamma$)
\[
 \gamma_0'(t_p) = g_{k_1}^{-1}\cdots g_{k_n}^{-1}.v_0 = g_{k_1}^{-1}\cdots g_{k_n}^{-1}.\gamma_0'(0) = h.\gamma_0'(0).
\]
Thus, $\gamma_0(mt_p) = h^m.\gamma_0(0)$ for any $m\in\Z$. It follows that
\[
 \gamma_0(+\infty) = \lim_{m\to\infty} h^m.\gamma_0(0) = w_a(h) 
\]
and
\[
 \gamma_0(-\infty) = \lim_{m\to-\infty} h^m.\gamma_0(0) = w_r(h).
\]
Hence $h$ is hyperbolic and $\mapeins(h) = \big( \gamma_0(+\infty), \gamma_0(-\infty) \big)$. Since $n$ is minimal, $h$ is primitive. The remaining statements follow immediately from the definitions of $\mapeinsn$ and $\mapdrein$.
\end{proof}

\begin{thm}\label{redprimitive}
Let $F\in \POQF(\Gamma_h)$. Then Algorithm~\ref{redalgo} finds all reduced $\Gamma$-forms that are equivalent to $F$. It first converts $F$ into an equivalent reduced $\Gamma$-form and then produces a cycle of pairwise distinct reduced $\Gamma$-forms which are all equivalent to $F$. Moreover, it determines the level of $F$ and the associated hyperbolic element in $\Gamma$.
\end{thm}

\begin{proof}
Suppose that $[\wh\gamma]\in\PGeo(\Gamma)/_\sim$ is the equivalence class of periodic geodesics on $X$ that is isomorphic to $F$.  Let $[\gamma_0]\in \UGeo/_\sim$ be the lift of $[\wh\gamma]$ such that 
\[
 \big(\gamma_0(+\infty), \gamma_0(-\infty)\big) = \mapzwei(F)
\]
and interpret Algorithm~\ref{redalgo} in terms of geodesics. Then Corollary~\ref{cor:preperiod} shows that the section `Calculation of preperiod' of Algorithm~\ref{redalgo} (lines 1--18) indeed yields a reduced lift of $[\wh\gamma]$ respectively a reduced $\Gamma$-form which is equivalent to $F$. More precisely, it shows that $h_m^{-1}.[\gamma_0]$ and $h_m^{-1}.F$ are reduced. Corollary~\ref{cor:period} proves that the section `Calculation of period' of Algorithm~\ref{redalgo} (lines 19--27) finds all reduced lifts of $[\wh\gamma]$ respectively all reduced $\Gamma$-forms equivalent to $F$. Moreover, if $n$ is the level of $F$, Corollary~\ref{cor:period} shows that $a_p^n$ is the hyperbolic element associated to $h_m^{-1}.F$. Hence, 
\[
h_ma_p^nh_m^{-1} = \mapvier(F).
\]
This completes the proof.
\end{proof}

\begin{thm}
Let $D>0$. Then there are only finitely many reduced $\Gamma$-forms with discriminant at most $D$.
\end{thm}

\begin{proof}
Let $F\in \POQF(\Gamma_h)$ and suppose that $g=\mapvier^{-1}(F)\in\Gamma_h$ is the associated hyperbolic element. Theorem~\ref{mqfm} shows that the discriminant of $F$ is
\[
 \discr(g) \sceq \tr(g)^2 - 4 = (a+d)^2 -4.
\]
Further, Corollary~\ref{cor:period} in combination with Proposition~\ref{allisos} yields that if $F$ is reduced then $g$ is of the form
\begin{equation}\label{standard}
 g= g_{k_1}^{-1}\cdots g_{k_n}^{-1}
\end{equation}
for unique elements $n\in\N_0$ and $k_1,\ldots, k_n\in\{1,\ldots, q-1\}$. Thus it suffices to show that there are only finitely many hyperbolic elements $g$ of the form \eqref{standard} with $\discr(g)\leq D$. For this, we will consider the growth of $\discr(g)$ as $n\to\infty$. 

Note that each element $g$ of the form \eqref{standard} is hyperbolic unless $g=g_1^{-m}$ or $g=g_{q-1}^{-m}$ for some $m\in\N_0$. For $k=1,\ldots, q-1$, let $h_k$ denote the representative of $g_k^{-1}$ in $\SL_2(\R)$ all of whose matrix entries are nonnegative (see \eqref{formulargk}). Suppose that 
\[
h=\textmat{a}{b}{c}{d}=h_{k_1} \cdots h_{k_n},
\]
but also allow $h$ to be non-hyperbolic (i.\,e., $h=h_1^m$ or $h=h_{q-1}^m$ is possible). Suppose $k\in\{1,\ldots, q-1\}$ such that $h h_k$ is hyperbolic. Then all entries of $h$ are nonnegative, $\tr(h)\geq 2$ and
\begin{align*}
 \tr(hg_k^{-1}) & = (a+b) \frac{\sin\left(\frac{k}{q}\pi\right)}{\sin\frac{\pi}{q}} + b \frac{\sin\left(\frac{k-1}{q}\pi\right)}{\sin\frac{\pi}{q}} + c\frac{\sin\left(\frac{k+1}{q}\pi\right)}{\sin\frac{\pi}{q}}
\\
& \geq \tr(h) \cdot \frac{\sin\left(\frac{k}{q}\pi\right)}{\sin\frac{\pi}{q}}
\\
& \geq \tr(h) + 2\left(\frac{\sin\left(\frac{k}{q}\pi\right)}{\sin\frac{\pi}{q}} - 1 \right).
\end{align*}
For $k\notin\{1,q-1\}$, 
\[
2\left(\frac{\sin\left(\frac{k}{q}\pi\right)}{\sin\frac{\pi}{q}} - 1 \right) \geq 2\left(\frac{\sin\left(\frac{2}{q}\pi\right)}{\sin\frac{\pi}{q}} - 1 \right) > 0,
\]
which shows that the trace of $h$ increases by at least
\[
 \delta\sceq 2\left(\frac{\sin\left(\frac{2}{q}\pi\right)}{\sin\frac{\pi}{q}} - 1 \right).
\]
Suppose that $k=1$ and let $m\in\N$. Then 
\begin{equation}\label{k1}
 \tr(hh_1^{m}) = \tr(h) + cm\lambda.
\end{equation}
We need to show that $c\geq 1$. To that end we note that for each $h_\ell$, $\ell=1,\ldots, q-1$, each matrix entry is of the form
\[
 \frac{\sin\left(\frac{a}{q}\pi\right)}{\sin\frac{\pi}{q}}  
\begin{cases}
 = 0 & \text{for $a\in\{0,q\}$}
\\
= 1 & \text{for $a\in\{1,q-1\}$}
\\
> 1 & \text{otherwise}
\end{cases}
\]
for some $a=a(\ell)\in \{0,\ldots, q\}$. If $h$ is not of the form $h'h_1$ for some $h'=h_{a_1}\cdots h_{a_p}$ with $p\in\N$, $a_1,\ldots, a_p\in\{1,\ldots, q-1\}$ then a straightforward induction shows that $c\geq 1$. In turn, 
\[
 \tr(h h_1^{m}) \geq \tr(h) + m\lambda.
\]
Further, if $\discr(hh_1^m) \leq D$ then 
\[
 m\leq \frac{\sqrt{D+4}}{\lambda}.
\]
Analogous results hold for $k=q-1$. 

Therefore, if $g$ is of the form \eqref{standard} with $\discr(g)\leq D$ then among the indices $k_1,\ldots, k_n$ there are at most $\sqrt{D+4}/\lambda$ appearances of $1$, $\sqrt{D+4}/\lambda$ appearances of $q-1$, and at most $\sqrt{D+4}/\delta$ indices other than $1$ and $q-1$. Thus, 
\[
 n\leq \frac{2\sqrt{D+4}}{\lambda} + \frac{\sqrt{D+4}}{\delta} \seqc n_0.
\]
Hence, there are at most $qn_0$ such elements $g$. This completes the proof.
\end{proof}

\section{The decision problem for Hecke triangle groups}\label{whenHecke}

The coefficients of the indefinite binary quadratic form $f = [1,1,-1]$ are obviously contained in $\Z[\lambda]$ for any $\lambda=\lambda_q = 2\cos\tfrac{\pi}{q}$, $q\in\N_{\geq 3}$. However, as we prove in Corollary~\ref{notinhecke} below, the projective oriented quadratic form $[f,+]$ is a $\Gamma_q$-form only for $q=3$, thus only for $\PSL_2(\Z)$. Therefore an algorithm is desirable which decides for any form $F\in\POQF$ with coefficients in $\Z[\lambda]$ whether or not $F$ is a $\Gamma$-form. 

As it is well-known, any fundamental domain for $\Gamma$ in $\h$ gives rise to such an algorithm. More precisely, it gives rise to an algorithm which solves the analogous problem on the level of (projective) matrices, i.\,e., it decides for any $g\in\PSL_2(\R)$ whether or not $g\in\Gamma$. Algorithm of this type are essentially based on tracking some point in $\h$. In this section we present an algorithm which takes advantage of the fundamental set $\wt{\mc F}$ for $\Gamma$ in $S\h$ (recall \eqref{funddomSH}), which allows us to deduce an upper estimate on its run time. This algorithm tracks two points in $P^1(\R)$, namely $0$ and $\infty$. 

For $g\in\PSL_2(\R)$ we define
\[
\sign(g) \sceq 
\begin{cases}
1 & \text{\begin{minipage}{8cm}some (and hence any) representative of $g$ in $\SL_2(\R)$ has at least one positive \textbf{and} one negative entry\end{minipage}}
\\[2mm]
0 & \text{otherwise.}
\end{cases}
\]
Further, if $x\in [g_k^{-1}.0,g_k^{-1}.\infty)$ for some $k\in\{1,\ldots, q-1\}$ then we set
\[
 \bdinterval(x) \sceq k.
\]

\begin{algorithm}
\caption{The decision algorithm}\label{algo_hecke}
\begin{algorithmic}[1]
\Statex \textbf{Input:}  $g\in \PSL_2(\R)$
\Statex
\Statex\Comment{Shift $g.0$ and $g.\infty$ into $[0,\infty)$.} 
\State $n \gets \min\{\ell\in\N_0 \mid \text{$T^\ell g.0>0$ and $T^\ell g.\infty>0$}\}$ 
\State $a_1 \gets T^n$
\State $h_1 \gets a_1g$ \Comment{We have $h_1.0>0$ and $h_1.\infty>0$.}
\Statex
\If{$h_1.0>h_1.\infty$} \Comment{Change order of $h_1.0$ and $h_1.\infty$ is necessary.}
\State $a_2\gets S$
\Else \State $a_2\gets \id$
\EndIf
\State $h_2\gets h_1a_2$ \Comment{We have $h_2.\infty>h_2.0\geq 0$.}
\Statex
\State $j\gets 2$ \Comment{Apply iteratively Proposition~\ref{distrC}\eqref{Cv} and \eqref{Ci}.}
\While{$h_j\not=\id$ and $\sign(h_j) = 0$} 
\State $k\gets \bdinterval(h_j.0)$
\State $a_{j+1}\gets g_k$
\State $h_{j+1}\gets a_{j+1}h_j$
\State $j\gets j+1$
\EndWhile
\Statex
\If{$h_j=\id$}
\State \textbf{return} $g=a_1^{-1}\cdots a_j^{-1}a_2^{-1}$ is in $\Gamma$. 
\Else \textbf{ return} $g$ is not in $\Gamma$.
\EndIf
\end{algorithmic}
\end{algorithm}

A combination of Proposition~\ref{distrC} with the fact that each $g_k^{-1}$, $k=1,\ldots, q-1$, has a representative in $\SL_2(\R)$ all of whose entries are nonnegative enables us to prove the correctness of Algorithm~\ref{algo_hecke}. To estimate the maximal run time of Algorithm~\ref{algo_hecke} let
\begin{equation}\label{gap}
 \delta_0\sceq1-\min_{a=1}^{q-2} \left\{ \frac{\sin\left(\frac{a+1}{q}\pi\right) + \sin\left(\frac{a}{q}\pi\right)}{\sin\frac{\pi}{q}} \right\}
\end{equation}
and note that $\delta_0>0$. For $p=\textmat{a}{b}{c}{d}\in \SL_2(\R)$ let 
\[
 \Sigma(p) \sceq a+b+c+d
\]
denote the sum of the matrix entries of $p$. For $g=\textbmat{a}{b}{c}{d}\in \PSL_2(\R)$ set 
\[
 \Sigma(g) \sceq |a+b+c+d|,
\]
which is the maximum of the entries sum of the two representatives of $g$ in $\SL_2(\R)$.

\begin{thm}\label{heckealgoOK}
For any $g\in\PSL_2(\R)$, Algorithm~\ref{algo_hecke} decides whether $g\in\Gamma$ or not. Let $h_2=h_2(g)$ denote the element in $\PSL_2(\R)$ obtained after performing lines~1-9. Then the while loop in lines~11-16 is repeated at most
\[
 m=\left\lceil \frac{\Sigma(h_2)}{\delta_0}\right\rceil
\]
times. Moreover, if $g\in\Gamma$ then Algorithm~\ref{algo_hecke} provides a presentation of $g$ as a word over the alphabet $\{g_1^{-1},\ldots, g_{q-1}^{-1}, S, T\}$.
\end{thm}

\begin{proof}
Let $g\in\PSL_2(\R)$ and consider the ordered pair $(g.0,g.\infty)$. In order to be able to engage Proposition~\ref{distrC} for deciding whether $g\in\Gamma$, we multiply $g$ with elements $a_1, a_2\in\Gamma$ to achieve that $h_2\sceq a_1ga_2$ satisfies 
\begin{equation}\label{prepare1}
h_2.\infty>h_2.0\geq0.
\end{equation}
In lines~1-3 of Algorithm~\ref{algo_hecke} we choose $a_1\sceq T^n$ to be the minimal nonnegative power of $T$ such that 
\[
 a_1g.0 = g.0+n\lambda \geq 0 \quad\text{and}\quad a_1g.\infty = g.\infty + n\lambda \geq 0.
\]
If $a_1g.0 > a_1g.\infty$ then we set $a_2\sceq S$ to achieve \eqref{prepare1}. If \eqref{prepare1} is already satisfied by $a_1g$ then we set $a_2\sceq \id$. This is done in lines~4-8 of Algorithm~\ref{algo_hecke}.

Suppose for a moment that $h_2\in\Gamma$. Then, by Proposition~\ref{distrC}\eqref{Ci}, 
\[
 h_2 = g_{k_1}^{-1}\cdots g_{k_m}^{-1}
\]
for unique elements $m\in\N_0$ and $k_1,\ldots, k_m\in\{1,\ldots, q-1\}$. Unless $m=0$, in which case $h_2=\id$, Proposition~\ref{distrC}\eqref{Cv} shows how to determine $k_1$. Consider
\[
 h_3\sceq g_{k_1}h_2 = g_{k_2}^{-1}\cdots g_{k_m}^{-1}.
\]
Note that $h_3.\infty > h_3.0\geq 0$. Thus, Proposition~\ref{distrC}\eqref{Cv} actually allows us to iteratively determine $k_1,\ldots, k_m$. For $j=2,\ldots, m+1$ let
\[
 h_j\sceq g_{k_{j-1}}\cdots g_{k_2}g_{k_1} h_2 = g_{k_j}^{-1}\cdots g_{k_m}^{-1}.
\]
Note that $h_{m+1} = \id$. Since each $g_k^{-1}$, $k=1,\ldots, q-1$, has a representative in $\SL_2(\R)$ all of whose entries are nonnegative (see \eqref{formulargk}),  also  $h_j$ for $j=2,\ldots, m+1$ is represented by a matrix in $\SL_2(\R)$ with  nonnegative entries only, thus $\sign(h_j) = 0$. 

This shows that if $h_2$ is indeed in $\Gamma$ then lines~10-16 of Algorithm~\ref{algo_hecke} detect the presentation of $h_2$ as given by Proposition~\ref{distrC}, determine correctly that $h_2\in\Gamma$ and return back a correct presentation of $g$ in line~18. 

We now show that Algorithm~\ref{algo_hecke} also termines if $g\notin\Gamma$ and that it detects this fact, and estimate simultaneously the maximal number of repetitions of the while loop.  For $k=1,\ldots, q-1$ let $p_k$ denote the representative of $g_k^{-1}$ in $\SL_2(\R)$ with nonnegative entries only (cf.\@ \eqref{formulargk}).  We consider the growth of the sums of the matrix entries of products of the form
\begin{equation}\label{product}
 p_{k_m}\cdots p_{k_1}
\end{equation}
as $m\to \infty$. For $s=\textmat{a}{b}{c}{d}\in \SL_2(\R)$ and any $k\in\{1,\ldots, q-1\}$ we have
\[
 \Sigma(p_ks)\sceq (a+b) \frac{\sin\left(\frac{k}{q}\pi\right)+\sin\left(\frac{k-1}{q}\pi\right)}{\sin\frac{\pi}{q}} + (c+d) \frac{\sin\left(\frac{k+1}{q}\pi\right) + \sin\left(\frac{k}{q}\pi\right)}{\sin\frac{\pi}{q}}.
\]
Note that 
\begin{align*}
&\max_{k=1}^{q-1} \left\{ \frac{\sin\left(\frac{k}{q}\pi\right)+\sin\left(\frac{k-1}{q}\pi\right)}{\sin\frac{\pi}{q}}, \frac{\sin\left(\frac{k+1}{q}\pi\right) + \sin\left(\frac{k}{q}\pi\right)}{\sin\frac{\pi}{q}}\right\} 
\\ &
\qquad
> 
\min_{k=1}^{q-1} \left\{ \frac{\sin\left(\frac{k}{q}\pi\right)+\sin\left(\frac{k-1}{q}\pi\right)}{\sin\frac{\pi}{q}}, \frac{\sin\left(\frac{k+1}{q}\pi\right) + \sin\left(\frac{k}{q}\pi\right)}{\sin\frac{\pi}{q}}\right\} 
= 1.
\end{align*}
Further note that for 
\[
 \mat{a'}{b'}{c'}{d'}\sceq p_n, \qquad (n\in\{1,\ldots, q-1\})
\]
we have
\begin{align*}
a'+b' &= \frac{\sin\left(\frac{n+1}{q}\pi\right) + \sin\left(\frac{n}{q}\pi\right)}{\sin\frac{\pi}{q}}
\intertext{and}
c'+d' & = \frac{\sin\left(\frac{n}{q}\pi\right)+\sin\left(\frac{n-1}{q}\pi\right)}{\sin\frac{\pi}{q}}.
\end{align*}
Thus 
\[
 \min\{ a'+b', c'+d'\} \geq 1.
\]
Suppose now that $s$ is of the form \eqref{product}. Straightforward induction yields
\[
 \min\{ a+b, c+d \} \geq 1.
\]
Therefore, with $\delta_0$ from \eqref{gap}, we find
\[
 \Sigma(p_ks) > \Sigma(s) + \delta_0. 
\]
Hence, $\Sigma(p_{k_m}\cdots p_{k_1}) \geq  m\delta_0$.

Suppose that $g\in\PSL_2(\R)$ with $\sign(g) = 0$. Let $\wt g$ be the representative of $g$ in $\SL_2(\R)$ with only nonnegative entries. In each repetition of the loop in lines~11-16 of Algorithm~\ref{algo_hecke}, $\Sigma(\wt g)$ gets decreased by at least $\delta_0$. Thus, after at most 
\[
 m = \left\lceil \frac{\Sigma(\wt g)}{\delta_0} \right\rceil
\]
repetitions of this loop, either the arising element $h_{j}$ equals $\id$ (and $g$ is detected to be in $\Gamma$) or it satisfies $\sign(h_j)=1$ (and $g$ is detected to not be in $\Gamma$). This completes the proof.
\end{proof}

\begin{cor}\label{notinhecke}
The projective oriented quadratic form $F=[f,+]$ with $f=[1,1,-1]$ is a $\PSL_2(\Z)$-form, but it is not a $\Gamma_q$-form for $q>3$.
\end{cor}

\begin{proof}
The element in $\PSL_2(\R)$ associated to $F$ by the isomorphism $\mapvier$ (see Theorem~\ref{mqfm}) is
\[
 g \sceq \bmat{1}{1}{1}{2},
\]
which obviously is an element of $\PSL_2(\Z)$. Algorithm~\ref{algo_hecke} shows that $g$ is not contained in $\Gamma_q$ for $q>3$. This completes the proof.
\end{proof}

\bibliography{ap_bib}
\bibliographystyle{amsplain}

\end{document}